\documentclass[12pt]{amsart}
\usepackage{enumerate,color,graphicx}
\usepackage{a4wide}
\allowdisplaybreaks
\usepackage[notcite,notref]{showkeys}

\let\pa\partial  
\let\na\nabla  
\let\eps\varepsilon  
  
\newcommand{\R}{{\mathbb R}} 
\newcommand{\diver}{\operatorname{div}}

\newcommand{\K}{\mathcal{K}} 

\newtheorem{theorem}{Theorem}   
\newtheorem{lemma}[theorem]{Lemma}

\newtheorem{definition}{Definition}

\begin{document}  

\title[Fractional cross-diffusion systems]{Analysis of a fractional 
cross-diffusion system \\ for multi-species populations}

\author[A.~J\"ungel]{Ansgar J\"ungel}

\address{Institute for Analysis and Scientific Computing, Vienna University of  
	Technology, Wiedner Hauptstra\ss e 8--10, 1040 Wien, Austria}
\email{juengel@tuwien.ac.at} 

\author[N.~Zamponi]{Nicola Zamponi}

\address{Institute for Analysis and Scientific Computing, Vienna University of  
	Technology, Wiedner Hauptstra\ss e 8--10, 1040 Wien, Austria}
\email{nicola.zamponi@tuwien.ac.at}

\date{\today}

\thanks{The authors have been partially supported by the Austrian Science Fund (FWF), 
grants P33010, F65, and W1245. This work has received funding from the European 
Research Council (ERC) under the European Union's Horizon 2020 research and 
innovation programme, ERC Advanced Grant no.~101018153.} 

\begin{abstract}
The global in time existence of weak solutions to a cross-diffusion system with
fractional diffusion in the whole space is proved. The equations describe
the evolution of multi-species populations in the regime of large-distance
interactions; they have been derived in the many-particle limit
from moderately interacting particle systems
with L\'evy noise. The existence proof is based
on a three-level approximation scheme, entropy and moment estimates, 
and a new Aubin--Lions compactness lemma in the whole space. 
\end{abstract}

% \paragraph{Keywords:}  
\keywords{Cross diffusion, fractional diffusion, population dynamics,
entropy method, global existence of solutions, moment estimates.}  
 
% \paragraph{AMS classification:}  
\subjclass[2000]{35K45, 35K55, 35Q92, 35R11.}

\maketitle

%%%%%%%%%%%%%%%%%%%%%%%%%%%%%%%%%%%%%%%%%%%%%%%%%%%%%%%%%%%%%%%%%%%%%%%%%%%%%%%

\section{Introduction}

The topic of this paper is the analysis of the following cross-diffusion system
with fractional derivatives, modeling the dynamics of multi-species populations:
\begin{align}
  & \pa_t u_i + \sigma_i(-\Delta)^\alpha u_i - \diver\bigg(\sum_{j=1}^n a_{ij}
	u_i\na(-\Delta)^{(\beta-1)/2} u_j\bigg) = 0\quad\mbox{in }\R^d,\ t>0, \label{1.eq} \\
	& u_i(0) = u_i^0\quad\mbox{in }\R^d, \ i=1,\ldots,n, \label{1.ic}
\end{align}
where $u_i(x,t)$ describes the population density of the $i$th species and $d\ge 2$
is the space dimension. The parameters are $\sigma_i\ge 0$ and $0<\alpha,\beta<1$.
The fractional Laplacian $(-\Delta)^s$ is defined for $0<s<1$ as the singular
integral operator
$$
  (-\Delta)^su(x) = c_{d,s}\int_{\R^d}\frac{u(x)-u(y)}{|x-y|^{d+2s}}dy, \quad
	\mbox{where }c_{d,s} = \frac{4^s\Gamma(d/2+s)}{\pi^{d/2}|\Gamma(-s)|},
$$
for $u\in H^s(\R^d)$. The integral is understood as the principal value, 
and $\Gamma$ denotes the Gamma function. The expression
$\na(-\Delta)^{(\beta-1)/2}$ can be interpreted as a fractional partial derivative
of order $\beta\in(0,1)$ and can be seen as a nonlocal gradient.

System \eqref{1.eq}--\eqref{1.ic} has been derived in \cite{DPR20}
as the many-particle limit
of the following interacting particle system driven by L\'evy noise:
\begin{equation}\label{1.levy}
  dX_i^{k,N}(t) = -\sum_{j=1}^n\frac{1}{N}\sum_{\ell=1}^{N}a_{ij}
	\na(-\Delta)^{(\beta-1)/2} V_N\big(X_i^{k,N}(t)-X_j^{\ell,N}(t)\big)dt 
	+ \sqrt{2\sigma_i}dL_i^k(t),
\end{equation}
where $i=1,\ldots,n$ and $k=1,\ldots,N$, 
$X_i^{k,N}(t)$ is the position of the $k$th particle of species $i$ at
time $t$, $V_N$ is a potential function, 
and $L_i^k$ is a L\'evy process of index $\alpha\in(0,1)$.
L\'evy processes include jumps and large-distance interactions instead of the
short-distance interactions of Brownian motion.
It was shown in \cite{DPR20} that if $V_N$ converges in the sense of
distributions to the delta distribution as $N\to\infty$, 
the empirical measures associated to \eqref{1.levy} converge in a certain sense 
to limiting processes with density $u_i$, solving \eqref{1.eq}--\eqref{1.ic}.
The global in time existence of strong solutions is proved in \cite{DPR20} for
sufficiently small initial data in the $H^s(\R^d)$ norm with $s>d/2$
in the regime $2\alpha>\beta+1$, in which
self-diffusion dominates cross-diffusion. In this paper, we prove the global in time
existence of weak solutions without any smallness
assumption on the initial data and for all values of $\alpha$, $\beta\in(0,1)$.
In particular, our proof allows for the case without self-diffusion, $\sigma_i=0$.
The key idea of our analysis is to exploit the entropy 
structure of \eqref{1.eq}, thus significantly extending the results of \cite{DPR20}. 

Most results on cross-diffusion systems in the literature refer to local models;
see, e.g., the references in \cite{Jue16}. Nonlocal cross-diffusion systems
have been investigated rather recently \cite{DiMo21,DiFa13,GaVe19,GHLP21,JPZ21}. 
Fractional diffusion was introduced in the Keller--Segel system
to model cellular population dispersal with anomalous diffusion \cite{Esc06}. 
Another application is a three-species food-chain cross-diffusion system
with fractional operators \cite{IWM21}. In this paper, we analyze 
the entropy structure of fractional cross-diffusion systems for the first time.

System \eqref{1.eq} can be seen as an extension of the local cross-diffusion system
of \cite{CDJ19}, which is formally obtained from \eqref{1.eq} by setting 
$\alpha=\beta=1$. The entropy structure of the local
model was investigated in \cite{JPZ21,JuZu20}. It turned out that such a structure
holds if there exist numbers $\pi_1,\ldots,\pi_n>0$ such that
$$
  \pi_i a_{ij} = \pi_j a_{ji}\quad\mbox{for }i,j=1,\ldots,n,
$$
and this condition is also assumed in this work.
It can be interpreted as the detailed-balance condition for the
Markov chain associated with $(a_{ij})$, and $(\pi_1,\ldots,\pi_n)$ is the
invariant measure. Together with the parabolicity condition of Petrovskii
\cite{Ama93}, i.e., all eigenvalues of $(a_{ij})\in\R^{n\times n}$ have a real
part, this implies that the matrix $(\pi_i a_{ij})\in\R^{n\times n}$ is
symmetric and positive definite \cite[Prop.~3]{ChJu21}. A formal computation
using a generalized Stroock--Varopoulos inequality (see Lemma \ref{lem.GSVI}
in the Appendix) shows that
\begin{equation}\label{1.ei0}
  \frac{d}{dt}H[u] + 4\sum_{i=1}^n\sigma_i\int_{\R^d}
	|(-\Delta)^{\alpha/2}\sqrt{u_i}|^2dx
	+ \lambda\sum_{i=1}^n\int_{\R^d}|\na(-\Delta)^{(\beta-1)/4} u_i|^2 dx \le 0,
\end{equation}
where $\lambda>0$ is the smallest eigenvalue of $(\pi_i a_{ij})$ and
\begin{equation}\label{1.ent}
  H[u] = \sum_{i=1}^n\pi_i\int_{\R^d}u_i\log u_i dx
\end{equation}
is the entropy functional. Taking into account the mass conservation and the
fractional Gagliardo--Nirenberg inequality, we obtain an estimate for
$u_i$ in $L^p(0,T;L^p(\R^d))$ for some $p>2$. Together with the
$L^2(0,T;H^{(\beta+1)/2}(\R^d))$ bound for $u_i$ from \eqref{1.ei0}, 
this is sufficient to handle the product $u_i\na(-\Delta)^{(\beta-1)/2}u_j$.

The mathematical difficulties to make these observations rigorous are of technical
nature. Indeed, since the fractional integral operator is singular,
we regularize the Riesz kernel $\K(x)=|x|^{1-\beta-d}$ of the Riesz potential
$(-\Delta)^{(\beta-1)/2}v = \K*u$ by some kernel $\K^{(\eps)}$ to define the
approximate scheme. Unfortunately, this distroys the
$L^2(0,T;H^{(\beta+1)/2}(\R^d))$ estimate for $u_i$ that is needed to obtain
a bound for $u_i$ in $L^p(0,T;L^p(\R^d))$ for some $p>2$
(observe that we allow for $\sigma_i=0$). One idea to remedy this
issue is to add the function 
$$
  \kappa g_0[u_i] = \kappa\bigg(u_i^2-\frac{e^{-|x|^2}}{\pi^{d/2}}
	\int_{\R^d}u_i^2 dx\bigg),
$$
to the equation, where $\kappa>0$ is the second approximation parameter. 
This function preserves the mass conservation property and it
provides an $L^1(0,T;L^1(\R^d))$ bound for $u_i^2(\log u_i)_+$, where $z_+=\max\{0,z\}$ 
(note that \eqref{1.ei0} is derived by using formally
the test function $\pi_i\log u_i$). However, 
in order to build an approximated solution, we need to replace $g_0$ by a 
bounded continuous mapping $L^2(\R^d)\to L^2(\R^d)$.
This forces us to introduce a third level of approximation, namely
$$
  \kappa g_\rho[u_i](x) 
	= \kappa u_i(x)(W_\rho*u_i)(x) - \frac{e^{-|x|^2}}{\pi^{d/2}}\int_{\R^d}
	u_i(y)(W_\rho*u_i)(y)dy,
$$
where $W_\rho$ is a mollifier with $\rho>0$ such that $g_\rho[u_i]\to g_0[u_i]$ 
a.e.\ in $\R^d$ as $\rho\to 0$. This yields, after the limit $\rho\to 0$, 
the desired estimate
for $u_i^2(\log u_i)_+$, thus in a Lebesgue space slightly better than $L^2(\R^d)$.

The de-regularization limits $\rho\to 0$, $\eps\to 0$, and $\kappa\to 0$
are based on suitable compactness lemmas. We state and prove
an Aubin--Lions-type compactness result in fractional Sobolev spaces
leading to strong convergence in the critical space $L^2(0,T;L^2(\R^d))$.
Compactness in the whole space $\R^d$ is achieved by controlling some moments
of $u_i$ and applying the compactness result in \cite[Lemma 1]{CGZ20};
see Lemma \ref{lem.comp} in the Appendix.

We summarize our hypotheses:
\begin{enumerate}
\item[(H1)] Data: $d\ge 2$, $\sigma_i\ge 0$, $a_{ij}\ge 0$, $\alpha\in(0,1)$, and 
$\beta\in(0,1)$.
\item[(H2)] Diffusion matrix: All eigenvalues of the matrix 
$A=(a_{ij})\in \R^{n\times n}$ have a positive real part, and the detailed-balance 
condition holds, i.e., there exist $\pi_1,\ldots,\pi_n>0$ such that 
$\pi_i a_{ij} = \pi_j a_{ji}$ for all $i,j=1,\ldots,n$.
\item[(H3)] Initial data: $u^0=(u_1^0,\ldots,u_n^0)$ satisfies
$u_i^0\in L^1(\R^d; (1+|x|^2)^{m/2}dx)$, $u_i^0\log u_i^0\in L^1(\R^d)$ for
$i=1,\ldots,n$, where $0<m<\min\{1,2\alpha\}$. 
\end{enumerate}

We already mentioned that Hypothesis (H2) implies that the matrix 
$(\pi_i a_{ij})\in\R^{n\times n}$ is symmetric and positive definite,
and hence, its eigenvalues are real and positive.
The moment assumption on $u_i$ in Hypothesis (H3)
is needed for the moment estimate of $u_i(t)$ that in turn is used to prove
the compactness in $\R^d$.

\begin{definition}
We say that $u=(u_1,\ldots,u_n):\R^d\times[0,\infty)\to\R^n$ is a {\em weak solution}
to \eqref{1.eq}--\eqref{1.ic} if 
\begin{align*}
& u_i\in L^2(0,T;H^{(\beta+1)/2}(\R^d)), \\
& u_i\in L^\infty(0,T;L^1(\R^d;(1+|x|^2)^{m/2}dx)
\quad\mbox{with $m>0$ as in (H3)}, \\
& \sqrt{u_i}\in L^2(0,T;H^\alpha(\R^d))\quad\mbox{if }\sigma_i>0,
\end{align*}
equation \eqref{1.eq} holds in the sense of $L^q(0,T;W^{-1,q}(\R^d))$ for some
$q>1$, and the initial condition \eqref{1.ic} holds in the sense of
$W^{-1,q}(\R^d)$. 	
\end{definition}
We show in Lemma \ref{lem.beta} below,
using the product rule for the fractional Laplacian \cite[Prop.~1.5]{BWZ17},
that $u_i\in L^\infty(0,T;L^1(\R^d))$
and $\sqrt{u_i}\in L^2(0,T;H^\alpha(\R^d))$ imply that
$(-\Delta)^{\alpha/2}u_i\in L^2(0,T;L^1(\R^d))$ such that the weak formulation
of \eqref{1.eq} makes sense. 

Our main result is as follows.

\begin{theorem}[Global existence]\label{thm.ex}
Let Hypotheses (H1)--(H3) hold. Then there exists a weak solution $u$
to \eqref{1.eq}--\eqref{1.ic}, which is nonnegative, i.e.\ $u_i(t)\ge 0$ a.e.\
in $\R^d$, conserves the mass,
$$
  \int_{\R^d}u_i(t)dx = \int_{\R^d}u_i^0 dx\quad\mbox{for }t>0,\ i=1,\ldots,n,
$$
and satisfies the entropy inequality,
\begin{align}
  \sum_{i=1}^n&\pi_i\int_{\R^d}u_i(t)\log u_i(t)dx 
	+ C\sum_{i=1}^n\sigma_i\int_0^t\int_{\R^d}\big|(-\Delta)^{\alpha/2}\sqrt{u_i}\big|^2
	dxds \label{1.ei} \\
	&{}+ \lambda\sum_{i=1}^n\int_0^t\int_{\R^d}|\na(-\Delta)^{(\beta-1)/4}u_i|^2 dxds
	\le \sum_{i=1}^n\pi_i\int_{\R^d}u_i^0\log u_i^0 dx, \quad t>0. \nonumber
\end{align}
\end{theorem}

{We discuss some generalizations and extensions. First, we expect that in the
limit $\alpha$, $\beta\to 1$, any solution to \eqref{1.eq}--\eqref{1.ic} converges 
strongly in $L^1(\R^d)$ towards a solution $\overline{u}$ to the limit problem
$$
	\pa_t \overline{u}_i - \sigma_i\Delta \overline{u}_i 
	- \diver\bigg(\sum_{j=1}^n a_{ij}\overline{u}_i\na \overline{u}_j\bigg) 
	= 0\quad\mbox{in }\R^d,\ t>0,
$$
satisfying the initial condition \eqref{1.ic}. Indeed, the uniform estimates
in Section \ref{sec.est} become stronger when $\alpha$, $\beta\to 1$
(see Lemma \ref{lem.beta} below), and
the limit in the equations appears to be rather standard.}

{Second, we may include reaction terms $f_i(u)$ on the right-hand side 
of \eqref{1.eq} 
satisfying suitable conditions that ensure the compatibility with the 
entropy structure. More precisely, we may assume that there exists a constant
$C>0$ such that for suitable functions $u$,
$$
  \int_{\R^d}\pi_i f_i(u)\log u_i dx \le C H[u] 
	+ C\sum_{i=1}^n\|u_i\|_{L^r(\R^d)}^2 + C,
$$
where $1\le r<2d/(d-\beta-1)$, or pointwise bounds on $f_i(u)\log u_i$ that
yield the previous integral estimate. This assumption ensures that the
contributions of the reaction terms are controlled by the entropy $H[u]$
via Gronwall's lemma and by the entropy production via $L^p(\R^d)$ interpolation
and fractional Sobolev embeddings. We also need the quasi-positivity condition
$f_i(u)\ge 0$ for $u_i=0$ to guarantee the nonnegativity of the solution as well as
the bound
$$
  \sum_{i=1}^n\int_{\R^d} f_i(u)(1+|x|^2)^{m/2}dx \le C\sum_{i=1}^n\int_{\R^d}
	u_i(1+|x|^2)^{m/2}dx + C
$$
(recall that $m>0$ is defined in Hypotheiss (H3))
to ensure the boundedness of the total mass and the $m$th moment of the solution.}

{Third, one may ask whether the analysis is more convenient using the Fourier
transform, since the fractional Laplacian and its inverse can be defined trough
this transform; see, e.g., \cite{NPV12,Ste70}. However, the presence of nonlinear
terms, and in particular the use of the test function $\log u_i$ in the entropy
method, complicates the use of the Fourier transform. For this reason, we
define the fractional Laplacian via a singular integral.}

{Fourth, Hypothesis (H2) about the positivity of the real parts of the eigenvalues
of $(a_{ij})$ is crucial for our analysis, since this condition ensures the
parabolicity of the problem in the sense of Petrovskii \cite{Ama93}. One may
ask what happens when the determinant of $(a_{ij})$ vanishes. 
A simple example has been studied
by Bertsch in \cite{BGHP85}, where $n=2$, $a_{ij}=1$ for $i,j=1,2$, and
$\alpha=0$, $\beta=1$. This example was generalized in \cite{DrJu20} to $n\ge 2$
species. An important feature of the model of \cite{BGHP85} is that the sum
$u_1+u_2$ solves a porous-medium equation with a quadratic term.
It was proved that segregated solutions at initial time
remain segregated for every time, which is essentially a consequence of the fact
that solutions to the porous-medium equation with compact support at initial time
remain compactly supported at any time. Since the fractional porous-medium equation
has a similar property \cite{CaVa11a}, we believe that a similar result as that
one in \cite{BGHP85} could be proved also for \eqref{1.eq}--\eqref{1.ic}.
Furthermore, the global existence of weak solutions to this problem with
$\sigma_i=\sigma$ and $a_{ij}=1$ appears to be in reach. 
Indeed, the sum $u_{\rm tot}:=\sum_{i=1}^n u_i$ satisfies the nonlocal
porous-medium equation
$$
  \pa_t u_{\rm tot} + \sigma(-\Delta)^\alpha u_{\rm tot}
	- \diver(u_{\rm tot}\na(-\Delta)^{(\beta-1)/2}u_{\rm tot}) = 0
	\quad\mbox{in }\R^d,\ t>0,
$$
which has been investigated in \cite{CaVa11,CaVa11a}, and the equation satisfied
by each single-species density $u_i$ can be seen as a (fractional) transport equation,
where $u_{\rm tot}$ plays the role of a given potential.}

The paper is organized as follows. We derive {\em formally} some a priori estimates
in Section \ref{sec.est}. Besides being an illustration of our strategy, the
computations will be used in the subsequent sections, in particular for the limit
procedure at the
last approximation level. The approximate problem with three approximation
levels is introduced in Section \ref{sec.approx}, and its global well-posedness
is proved. In Section \ref{sec.limit}, the limit in the approximate problem is shown.
Finally, we collect some technical results and prove an Aubin--Lions-type compactness
lemma in the Appendix.

\subsection*{Notation}

The space $W^{s,p}(\R^d)$ with $s>0$ and $1\le p\le \infty$ is the usual
fractional Sobolev space; we set $H^s(\R^d)=W^{s,2}(\R^d)$. 
We write $\|\cdot\|_p$ for the norm in $L^p(\R^d)$, $1\le p\le\infty$,
and we define for $m>0$ the space
$$
  L^1(\R^d; (1+|x|^2)^{m/2}dx) = \bigg\{v\in L^1(\R^d):
	\int_{\R^d}v(x)(1+|x|^2)^{m/2}dx<\infty\bigg\}.
$$
The characteristic function on a set $B\subset\R^d$ is written as
$\mathrm{1}_B$. Finally,
we denote by $C>0$ a generic constant whose value may change from line to line.

%%%%%%%%%%%%%%%%%%%%%%%%%%%%%%%%%%%%%%%%%%%%%%%%%%%%%%%%%%%%%%%%%%%%%%%%%%%%%%%

\section{A priori estimates}\label{sec.est}

In this section, we derive {\em formally} some a priori estimates.
First, we provide a proof for the entropy inequality \eqref{1.ei0}.

\begin{lemma}
Let $u$ be a smooth solution to \eqref{1.eq}--\eqref{1.ic}. Then
\begin{equation}\label{2.ei}
  \frac{d}{dt}H[u] + 4\sum_{i=1}^n\sigma_i\int_{\R^d}
	|(-\Delta )^{\alpha/2}\sqrt{u_i}|^2dx
	+ \lambda\sum_{i=1}^n\int_{\R^d}|\na(-\Delta)^{(\beta-1)/4}u_i|^2 dx \le 0,
\end{equation}
where the entropy $H[u]$ is defined in \eqref{1.ent} and
$\lambda>0$ is the smallest eigenvalue of $(\pi_i a_{ij})\in\R^{n\times n}$.
\end{lemma}

\begin{proof}
Using $\log u_i$ {\em formally} as a test function in \eqref{1.eq} yields
\begin{equation}\label{2.aux1}
  \frac{d}{dt}H[u] = -\sum_{i=1}^n\sigma_i\int_{\R^d}\log u_i(-\Delta)^\alpha u_i dx
	- \sum_{i,j=1}^n\pi_i a_{ij}\int_{\R^d}\na u_i\cdot\na(-\Delta)^{(\beta-1)/2}u_j dx.
\end{equation}
We integrate by parts in the last integral and use the positive definiteness
of the matrix $(\pi_i a_{ij})$ to obtain
\begin{align*}
   \sum_{i,j=1}^n\pi_i a_{ij}\int_{\R^d}\na u_i\cdot\na(-\Delta)^{(\beta-1)/2}u_j dx
	&=  \sum_{i,j=1}^n\pi_i a_{ij}\int_{\R^d}\na(-\Delta)^{(\beta-1)/4} u_i
	\cdot\na(-\Delta)^{(\beta-1)/4}u_j dx \\
	&\ge \lambda\sum_{i=1}^n\int_{\R^d}|\na(-\Delta )^{(\beta-1)/4}u_i|^2dx.
\end{align*}
We apply the generalized Stroock--Varopolous inequality in Lemma \ref{lem.GSVI}
(see Appendix \ref{sec.aux})
to the first integral on the right-hand side of \eqref{2.aux1} to conclude the proof.
\end{proof}

The mass conservation and entropy inequality \eqref{2.ei} yield the following 
bounds for $i=1,\ldots,n$:
\begin{equation}\label{2.est1}
  \|u_i\|_{L^\infty(0,\infty;L^1(\R^d))} 
	+ \sigma_i\|\sqrt{u_i}\|_{L^2(0,T;H^\alpha(\R^d))} 
	+ \lambda\|u_i\|_{L^2(0,T;H^{(\beta+1)/2}(\R^d))} \le C.
\end{equation}

We derive further a priori estimates from the entropy inequality and the
Gagliardo--Nirenberg inequality.

\begin{lemma}\label{lem.beta}
Let $u$ be a smooth solution to \eqref{1.eq}--\eqref{1.ic}. Then there
exists a constant $C>0$, not depending on $u$, such that for $i=1,\ldots,n$,
\begin{equation}\label{2.est2}
  \|\na(-\Delta)^{(\beta-1)/2}u_i\|_{L^2(0,T;L^{2d/(d+\beta-1)}(\R^d))}
	+ \sqrt{\sigma_i}\|(-\Delta)^{\alpha/2}u_i\|_{L^2(0,T;L^1(\R^d))} \le C.
\end{equation}
Moreover, there exist $p^*,q^*>1$ such that for $i,j=1,\ldots,n$,
\begin{equation}\label{2.est3}
  \|u_i\na(-\Delta)^{(\beta-1)/2}u_j\|_{L^{q^*}(0,T;L^{p^*}(\R^d))} \le C.
\end{equation}
\end{lemma}

\begin{proof}
We claim that
\begin{equation}\label{2.Lqp}
  \|u_i\|_{L^q(0,T;L^p(\R^d))} \le C, \quad\mbox{where }q\ge 2\mbox{ and } 
	\frac{1}{p} + \frac{1}{q}\bigg(1+\frac{\beta+1}{d}\bigg) = 1.
\end{equation}
By interpolation, it holds for $1<p,q<\infty$ and $0<\theta<1$ with
$1/p=\theta/r+(1-\theta)$ that
$$
  \|u_i\|_{L^q(0,T;L^p(\R^d))}^q 
	= \int_0^T\|u_i\|_p^q dt 
	\le \int_0^T\|u_i\|_r^{\theta q}\|u_i\|_1^{(1-\theta)q}dt.
$$
Taking into account estimate \eqref{2.est1} and the fractional Sobolev embedding
$H^{(\beta+1)/2}(\R^d)$ $\hookrightarrow L^r(\R^d)$ for $r=2d/(d-\beta-1)$
\cite[Theorem 6.5]{NPV12} and choosing $\theta=2/q\le 1$, we find that
$$
  \|u_i\|_{L^q(0,T;L^p(\R^d))}^q 
	\le C\|u_i\|_{L^\infty(0,T;L^1(\R^d))}^{(1-\theta)q}\int_0^T
	\|u_i\|_{H^{(\beta+1)/2}(\R^d)}^{\theta q}dt \le C.
$$
Inserting $\theta=2/q$ and $1/r=1/2-(\beta+1)/(2d)$, we have
$$
  \frac{1}{p} = \frac{\theta}{r} + (1-\theta)
	= \frac{2}{q}\bigg(\frac12 - \frac{\beta+1}{2d}\bigg) + 1 - \frac{2}{q}
	= 1 - \frac{1}{q}\bigg(1 + \frac{\beta+1}{d}\bigg),
$$
which proves the claim. Choosing $q=p$ in \eqref{2.Lqp} yields
$$
  \|u_i\|_{L^p(0,T;L^p(\R^d))} \le C\quad\mbox{for }p=2 + \frac{\beta+1}{d} > 2.
$$
It follows from the Hardy--Littlewood--Sobolev inequality (Lemma \ref{lem.HLS}
with $s=(1-\beta)/4$) that
\begin{equation*}
  \|\na(-\Delta)^{(\beta-1)/2}u_j\|_{2d/(d+\beta-1)}
	\le C\|\na(-\Delta)^{(\beta-1)/4}u_j\|_2
\end{equation*}
and therefore, because of \eqref{2.est1},
\begin{equation}\label{2.beta}
   \|\na(-\Delta)^{(\beta-1)/2}u_j\|_{L^2(0,T;L^{2d/(d+\beta-1)}(\R^d))} \le C.
\end{equation}
Since $2d/(d+\beta-1)>2$, the product $u_i\na(-\Delta)^{(\beta-1)/2}u_j$ is bounded
in some $L^{q^*}(0,T;$ $L^{p^*}(\R^d))$ for suitable $q^*,p^*>1$.

It remains to derive the bound for $(-\Delta)^{\alpha/2}u_i$. By the product rule
for the fractional Laplacian \cite[Prop.~1.5]{BWZ17},
$$
  (-\Delta)^{\alpha/2}u_i(x) = 2\sqrt{u_i}(x)(-\Delta)^{\alpha/2}\sqrt{u_i}(x)
	- c_{d,\alpha/2}\int_{\R^d}\frac{(\sqrt{u_i(x)}-\sqrt{u_i(y)})^2}{|x-y|^{d+\alpha}}dy.
$$
We take the $L^1(\R^d)$ norm and use the Cauchy--Schwarz inequality to find that
\begin{align*}
  \|(-\Delta)^{\alpha/2}u_i\|_1
	&\le 2\|\sqrt{u_i}\|_2\|(-\Delta)^{\alpha/2}\sqrt{u_i}\|_2
	+ c_{d,\alpha/2}\int_{\R^d}\int_{\R^d}
	\frac{(\sqrt{u_i(x)}-\sqrt{u_i(y)})^2}{|x-y|^{d+\alpha}}dxdy \\
	&= 2\|u_i\|_1^{1/2}\|(-\Delta)^{\alpha/2}\sqrt{u_i}\|_2
	+ 2\int_{\R^d}\sqrt{u_i}(-\Delta)^{\alpha/2}\sqrt{u_i}dx \\
	&\le 4\|u_i\|_1^{1/2}\|(-\Delta)^{\alpha/2}\sqrt{u_i}\|_2.
\end{align*}
After taking the square and integrating over time, we obtain
$$
  \sigma_i\|(-\Delta)^{\alpha/2}u_i\|_{L^2(0,T;L^1(\R^d))}^2
	\le 4\sigma_i\|u_i\|_{L^\infty(0,T;L^1(\R^d))}
	\|(-\Delta)^{\alpha/2}\sqrt{u_i}\|_{L^2(0,T;L^2(\R^d))}^2 \le C.
$$
This ends the proof.
\end{proof}

Next, we derive some moment bounds for $u_i$.

\begin{lemma}\label{lem.moment}
Let $u$ be a smooth solution to \eqref{1.eq}--\eqref{1.ic} and $0<m<\min\{1,2\alpha\}$. 
Then there exists a constant $C>0$, independent of $u$, such that
$$
  \sup_{0<t<T}\int_{\R^d}u_i(x,t)(1+|x|^2)^{m/2}dx \le C(T), \quad i=1,\ldots,n.
$$
\end{lemma}

\begin{proof}
We use {\em formally} the test function $(1+|x|^2)^{m/2}$ with $0<m<\min\{1,2\alpha\}$ 
in \eqref{1.eq}:
\begin{align}\label{2.dudt}
  \frac{d}{dt}\int_{\R^d}u_i(1+|x|^2)^{m/2}dx
	&= -\sigma_i\int_{\R^d}u_i(-\Delta)^\alpha(1+|x|^2)^{m/2}dx \\
	&\phantom{xx}{}
	- \sum_{j=1}^n a_{ij}\int_{\R^d}u_i\na(1+|x|^2)^{m/2}\cdot\na(-\Delta)^{(\beta-1)/2}
	u_j dx. \nonumber
\end{align}
To estimate the first term on the right-hand side, we claim 
that there exists $C>0$ such that 
\begin{equation}\label{2.xm}
  |(-\Delta)^\alpha(1+|x|^2)^{m/2}|\le C(1+|x|^2)^{m/2}\quad\mbox{ for all }x\in\R^d.
\end{equation}
Indeed, we infer from \cite[Lemma 3.2]{NPV12} that 
\begin{align*}
  -(-\Delta)^\alpha(1+|x|)^{m/2} &= \frac{c_{d,\alpha}}{2}\int_{\R^d}
	\frac{(1+|x+y|^2)^{m/2}+(1+|x-y|^2)^{m/2}-2(1+|x|^2)^{m/2}}{|y|^{d+2\alpha}}dy \\
	&=: I_1 + I_2,
\end{align*}
where
\begin{align*}
	I_1 &= \frac{c_{d,\alpha}}{2}\int_{\{|y|>1\}}
	\frac{(1+|x+y|^2)^{m/2}+(1+|x-y|^2)^{m/2}-2(1+|x|^2)^{m/2}}{|y|^{d+2\alpha}}dy, \\
	I_2 &= \frac{c_{d,\alpha}}{2}\int_{\{|y|\le 1\}}
	\frac{(1+|x+y|^2)^{m/2}+(1+|x-y|^2)^{m/2}-2(1+|x|^2)^{m/2}}{|y|^{d+2\alpha}}dy.
\end{align*}
The triangle inequality implies that
$$
  |I_1| \le C\int_{\{|y|>1\}}\frac{|x|^m+|y|^m}{|y|^{d+2\alpha}}dy 
	\le C(1+|x|^2)^{m/2},
$$
since the integrability is ensured if $m-d-2\alpha<-d$ or, equivalently, $m<2\alpha$.
The function $\Phi_x(y):=(1+|x+y|^2)^{m/2}+(1+|x-y|^2)^{m/2}-2(1+|x|^2)^{m/2}$
satisfies $\Phi_y(0)=|\na_y\Phi_y(0)|=0$ and
$$
  |D^2_y\Phi_x(y)| \le C(1 + |x+y|^2)^{m/2-1} + C(1 + |x-y|^2)^{m/2-1}, \quad
	x,y\in\R^d,
$$
and this expression is bounded for all $x,y\in\R^d$. 
We infer from Taylor's theorem that
$|\Phi_x(y)|=\frac12|D^2_y\Phi_x(\theta y)||y|^2 \le C|y|^2$ for $y\in\R^d$,
where $\theta\in[0,1]$ is a suitable number. Therefore,
$|I_2|\le C\int_{\{|y|\le 1\}}|y|^{2-d-2\alpha}dy\le C\le C(1+|x|^2)^{m/2}$, 
since $\alpha<1$. This shows the claim.

We estimate the last term in \eqref{2.dudt}. Choosing $p=2d/(d-\beta+1)$ in
\eqref{2.Lqp}, we find that
$$
  \|u_i\|_{L^q(0,T;L^{2d/(d-\beta+1)}(\R^d))} \le C, \quad\mbox{where }
	q = \frac{2(d+\beta+1)}{d+\beta-1} > 2.
$$
Because of \eqref{2.Lqp} and \eqref{2.beta}, the product 
$u_i\na(-\Delta)^{(\beta-1)/2}u_j$ can be estimated according to
$$
  \|u_i\na(-\Delta)^{(\beta-1)/2}u_j\|_{L^r(0,T;L^1(\R^d))} \le C
	\quad\mbox{for some }r>1.
$$
Taking into account that $\na(1+|x|^2)^{m/2}$ is bounded in $\R^d$ if $m<1$, we obtain
$$
  -\sum_{j=1}^n a_{ij}\int_0^t\int_{\R^d}u_i\na(1+|x|^2)^{m/2}
	\cdot\na(-\Delta)^{(\beta-1)/2}	u_j dx d\tau \le C.
$$
Summarizing, we conclude from \eqref{2.dudt} that
$$
  \int_{\R^d}u_i(t)(1+|x|^2)^{m/2}dx
	\le \int_{\R^d}u_i^0(1+|x|^2)^{m/2}dx + C\int_0^t\int_{\R^d}u_i
	(1+|x|^2)^{m/2}dxds + C
$$
for some $C>0$, which shows the result after applying Gronwall's lemma.
\end{proof}

\begin{lemma}
Let $u$ be a smooth solution to \eqref{1.eq}--\eqref{1.ic}. Then there exist
constants $C>0$ and $p>1$, independent of $u$, such that
$$
  \|\pa_t u_i\|_{L^p(0,T;W^{-1,p}(\R^d))} \le C.
$$
\end{lemma}

\begin{proof}
It follows from estimates \eqref{2.est2} and \eqref{2.est3} that there exists
$p>1$ such that
\begin{align*}
  \|\pa_t u_i\|_{L^p(0,T;W^{-1,p}(\R^d))}
	&\le \sigma_i\|(-\Delta)^{\alpha/2}u_i\|_{L^p(0,T;L^p(\R^d))} \\
	&\phantom{xx}{}
	+ \sum_{j=1}^n a_{ij}\|u_i\na(-\Delta)^{(\beta-1)/2}u_j\|_{L^p(0,T;L^p(\R^d))}
	\le C,
\end{align*}
which finishes the proof.
\end{proof}

%%%%%%%%%%%%%%%%%%%%%%%%%%%%%%%%%%%%%%%%%%%%%%%%%%%%%%%%%%%%%%%%%%%%%%%%%%%%%%%

\section{Approximate scheme}\label{sec.approx}

We approximate equation \eqref{1.eq} by introducing three approximation levels.
First, we regularize the Riesz potential. Noting that
$(-\Delta)^{(\beta-1)/2}u=\K_{(1-\beta)/2}*u$, where 
$\K_{(1-\beta)/2}(x) = |x|^{1-\beta-d}$ for $x\in\R^d$ is the kernel of the
Riesz potential, we define the
approximation $\K_s^{(\eps)}$ of $\K_s$ by
\begin{align}\label{2.defK}
  & \K_s^{(\eps)} := \widetilde\K_{s/2}^{(\eps)}*\widetilde\K_{s/2}^{(\eps)},
	\quad\mbox{where }\widetilde\K_{s/2}^{(\eps)}\in C_0^2(\R^d), \\
	& 0\le \widetilde\K_{s/2}^{(\eps)}\le \widetilde\K_{s/2}^{(\eps')}\le
	\K_{s/2}\quad\mbox{in }\R^d\mbox{ for }0<\eps'<\eps, \nonumber \\
	& \widetilde\K_{s/2}^{(\eps)}(x) = \K_{s/2}(x)\quad\mbox{for }\eps\le|x|\le 1/\eps.
	\nonumber
\end{align} 
Since $\int_{\R^d}u \K_s^{(\eps)}dx$ generally does not preserve the nonnegativity
for $u\ge 0$, we define $\K_s^{(\eps)}$ as a ``convolution square'' to guarantee
this property. Second, we introduce the mollifier
\begin{align*}
  & W_\rho(x) := \rho^{-d}W_1(x/\rho)\quad\mbox{for }x\in\R^d,\quad\mbox{where } \\
	& W_1\in C_0^0(\R^d),\quad W_1\ge 0\mbox{ in }\R^d,\quad \|W_1\|_1=1,
\end{align*}
that satisfies $W_\rho*u\to u$ a.e.\ in $\R^d$,
and the mapping $g_\rho:L^2(\R^d)\to L^2(\R^d)\cap L_0^1(\R^d)$,
$$
  g_\rho[u](x) := u(x)(W_\rho*u)(x) - \frac{e^{-|x|^2}}{\pi^{d/2}}\int_{\R^d}
	u(y)(W_\rho*u)(y)dy, \quad u\in L^2(\R^d),
$$
where $L_0^1(\R^d)$ is the space of $L^1(\R^d)$ functions with vanishing average.
This mapping satisfies the following properties:
\begin{align}
  & \|g_\rho[u]\|_1 \le 2\|u\|_2^2, \quad \|g_\rho[u]\|_2 \le C(\rho)\|u\|_2^2, 
	\label{3.g1} \\
	& \|g_\rho[u]-g_\rho[v]\|_2 \le C(\rho)\|u+v\|_2\|u-v\|_2, \label{3.g2}
\end{align}
where $u,v\in L^2(\R^d)$. These inequalities follow from the Young convolution
inequality, 
\begin{align*}
  \|u(W_\rho*u)\|_1 &\le \|u\|_2\|W_\rho*u\|_2
  \le \|u\|_2^2\|W_\rho\|_1 = \|u\|_2^2, \\
	\|u(W_\rho*u)\|_2 &\le \|u\|_2\|W_\rho*u\|_\infty
  \le \|u\|_2^2\|W_\rho\|_2 \le C(\rho)\|u\|_2^2.
\end{align*}
As explained in the introduction, 
the function $g_\rho$ is needed to obtain an $L^2\log L^2$
estimate, which is used to obtain strong convergence of the sequence of approximate 
solutions in $L^2(\R^d)$.
Furthermore, we add a Laplacian to \eqref{1.eq}. 
This leads to the approximate problem
\begin{align}
  & \pa_t u_i^{(\rho,\eps,\kappa)} - \kappa\Delta u_i^{(\rho,\eps,\kappa)}
	+ \sigma_i(-\Delta)^\alpha u_i^{(\rho,\eps,\kappa)} 
	+ \kappa g_\rho[u_i^{(\rho,\eps,\kappa)}] \label{3.eq} \\
	&\phantom{xx}{}= \diver\bigg(\sum_{j=1}^n a_{ij}(u_i^{(\rho,\eps,\kappa)})_+
	\na\K_{(1-\beta)/2}^{(\eps)}*u_j^{(\rho,\eps,\kappa)}\bigg), \nonumber \\
	& u_i^{(\rho,\eps,\kappa)}(0) = u_i^0\quad\mbox{in }\R^d,\ i=1,\ldots,n,
	\label{3.ic}
\end{align}
where $z_+=\max\{0,z\}$ denotes the positive part of $z\in\R$.

\subsection{Local well-posedness of the approximate problem}

We prove the existence of a local solution to \eqref{3.eq}--\eqref{3.ic}
by applying Banach's fixed-point theorem.
To this end, we introduce for $R>2\|u_0\|_2$ and $T>0$ the space
$$
  X_{R,T} := \big\{v\in C^0([0,T];L^2(\R^d)):\|v_i\|_{L^\infty(0,T;L^2(\R^d))}
	\le R,\ i=1,\ldots,n\big\}
$$
and the fixed-point mapping $F:X_{R,T}\to X_{R,T}$, $F(v)=u$, where 
$u=(u_1,\ldots,u_n)$ is the unique solution to the linear problem
\begin{align}
  & \pa_t u_i - \kappa\Delta u_i + \sigma_i(-\Delta)^\alpha u_i
	= -\kappa g_\rho[v_i] +	\diver\bigg(\sum_{j=1}^n a_{ij}(v_i)_+
	\na\K_{(1-\beta)/2}^{(\eps)}*v_j\bigg), \label{3.lin} \\
	& u_i(0) = u_i^0\quad\mbox{in }\R^d,\ i=1,\ldots,n. \nonumber
\end{align}
Since the kernel is regularized, this problem has a unique solution 
$u=(u_1,\ldots,u_n)$ with $u_i\in L^2(0,T;H^1(\R^d))$, 
$\pa_t u_i\in L^2(0,T;H^{-1}(\R^d))$, implying that 
$u_i\in C^0([0,T];L^2(\R^d))$.

We show that the mapping $F$ is well defined. We use the test function
$u_i$ in the weak formulation of \eqref{3.lin} and take into account
\eqref{3.g1}:
\begin{align*}
  \frac12\|u_i(t)&\|_2^2 + \kappa\int_0^t\|\na u_i\|_2^2ds
	+ \sigma_i\int_0^t\|(-\Delta)^{\alpha/2}u_i\|_2^2ds \\
	&\le \frac12\|u_i^0\|_2^2 +	C(\kappa,\rho)\int_0^t\|v_i\|_2^2\|u_i\|_2 ds \\
	&\phantom{xx}{}+ C\sum_{j=1}^n\int_0^t\|(v_i)_+\na\K_{(1-\beta)/2}^{(\eps)}*v_j\|_2
	\|\na u_i\|_2 ds.
\end{align*}
We apply the Young (convolution) inequality to obtain for $0<t<T$,
\begin{align*}
  \|u_i(t)&\|_2^2 + \kappa\int_0^t\|\na u_i\|_2^2ds
	+ \sigma_i\int_0^t\|(-\Delta)^{\alpha/2}u_i\|_2^2ds \\
	&\le 2\|u_i^0\|_{L^2(\R^d)}^2 + C(\kappa,\rho)\int_0^t\|v_i\|_2^4 ds \\
	&\phantom{xx}{}+ C(\kappa)\sum_{j=1}^n\int_0^t\|(v_i)_+\|_2^2
	\|\na\K_{(1-\beta)/2}^{(\eps)}\|_2^2\|v_j\|_2^2 ds \\
	&\le 2\|u_i^0\|_{L^2(\R^d)}^2 + C(\eps,\kappa,\rho)T\sum_{j=1}^n
	\|v_j\|_{L^\infty(0,T;L^2(\R^d))}^4.
\end{align*}
Therefore, since $\|u_i^0\|_2 < R/2$,
if $T>0$ is sufficiently small, we infer that $u\in X_{R,T}$,
proving the well-posedness of $F$. 

Next, we show that $F$ is a contraction on $X_{R,T}$. Let $v,v'\in X_{R,T}$
and set $u=F(v)$, $u'=F(v')$. The test function $u_i-u_i'$ in the weak
formulation of 
\begin{align*}
  \pa_t&(u_i-u_i') - \kappa\Delta(u_i-u_i') + \sigma_i(-\Delta)^\alpha(u_i-u_i')
	+ \kappa(g_\rho[v_i]-g_\rho[v_i']) \\
	&= \diver\bigg(\sum_{j=1}^n a_{ij}\big[\big((v_i)_+ - (v_i')_+\big)
	\na\K_{(1-\beta)/2}^{(\eps)}*v_j + (v_i')_+\na\K_{(1-\beta)/2}^{(\eps)}*(v_j-v_j')
	\big]\bigg)
\end{align*}
leads, after similar computations as before and using \eqref{3.g2}, for $0<t<T$, to
\begin{align*}
 \|(&u_i-u_i')(t)\|_2^2 + \kappa\int_0^t\|\na(u_i-u_i')\|_2^2 ds
  + \sigma_i\int_0^t\|(-\Delta)^{\alpha/2}(u_i-u_i')\|_2^2 ds \\
  &\le C(\rho)\int_0^t\|v_i+v_i'\|_2^2\|v_i-v_i'\|_2^2 ds
  + C(\kappa)\sum_{j=1}^n\int_0^t\big(\|(v_i)_+ - (v_i')_+\|_2^2
  \|\na\K_{(1-\beta)/2}^{(\eps)}*v_j\|_\infty^2 \\
	&\phantom{xx}{}
	+ \|(v_i')_+\|_2^2\|\na\K_{(1-\beta)/2}^{(\eps)}*(v_j-v_j')\|_\infty^2\big)ds \\
	&\le C(\eps,\kappa,\rho)\sum_{j,k=1}^n\int_0^t\|v_j-v_j'\|_2^2
	\big(\|v_k\|_2^2 + \|v_k'\|_2^2\big)ds \\
	&\le C(\eps,\kappa,\rho,R)T\sum_{j=1}^n\|v_j-v_j'\|_{L^\infty(0,T;L^2(\R^d))}^2.
\end{align*}
Hence, if $T>0$ is sufficiently small, $F$ is a contraction on $X_{R,T}$.
We conclude from Banach's fixed-point theorem that there exists $T^*>0$ and a unique
fixed point $u\in X_{R,T^*}$ of $F$, i.e.\ a unique solution
$u^{(\eps,\kappa,\rho)}\in L^2(0,T^*;H^1(\R^d))$ with $\pa_t u_i^{(\eps,\kappa,\rho)}
\in L^2(0,T^*;H^{-1}(\R^d))$ to \eqref{3.eq}--\eqref{3.ic}.

\subsection{Uniform bounds and global well-posedness}\label{sec.global}

We show that the solution $u=u^{(\eps,\kappa,\rho)}\in C^0([0,T^*];L^2(\R^d))$, 
derived in the previous subsection, is actually global in time. 
First, we prove that $u_i(t)\ge 0$ for $t\in[0,T^*]$. We use the test function
$(u_i)_-=\min\{0,u_i\}$ as a test function in the weak formulation of \eqref{3.eq}:
\begin{align*}
  \frac12&\int_{\R^d}(u_i)_-^2(t) dx
	+ \kappa\int_0^t\int_{\R^d}|\na(u_i)_-|^2 dxds
	= -\sigma_i\int_0^t\int_{\R^d}(u_i)_-(-\Delta)^{\alpha}u_i dxds \\
  &\phantom{xx}{}- \kappa\int_0^t\int_{\R^d}(u_i)_-g_\rho[u_i]dxds
	- \sum_{j=1}^na_{ij}\int_0^t\int_{\R^d}(u_i)_+(\na\K_{(1-\beta)/2}^{(\eps)}*u_j)
	\cdot\na(u_i)_- dxds \\
	&=: I_3+I_4+I_5.
\end{align*}
Since $(u_i)_+\na(u_i)_-=(u_i)_+\mathrm{1}_{\{u_i<0\}}\na u_i=0$, 
we have $I_5=0$. Moreover, by a symmetry argument (also see \cite[Lemma 7.4]{BKM10}),
\begin{align*}
  I_3 &= -\frac{\sigma_i c_{d,\alpha}}{2}\int_{\R^d}\int_{\R^d}
	\frac{[(u_i(x))_- - (u_i(y))_-](u_i(x)-u_i(y))}{|x-y|^{d+2\alpha}}dxdy \le 0, 
	\quad\mbox{and} \\
  I_4 &\le -\int_0^t\int_{\R^d}u_i(W_\rho*u_i)(u_i)_- dxds
	= -\int_0^t\int_{\R^d}(u_i)_-^2(W_\rho*u_i) dxds \\
	&\le \int_0^t\|W_\rho*u_i\|_\infty\|(u_i)_-\|_2^2 ds 
	\le C(\rho)\int_0^t\|u_i\|_2\|(u_i)_-\|_2^2 ds.
\end{align*}
We conclude that for $0<t<T^*$,
$$
  \|(u_i)_-^2(t)\|_2^2
	\le C(\rho)\int_0^t\|u_i\|_2\|(u_i)_-\|_2^2 ds.
$$
Since $t\mapsto\|u_i(t)\|_2$ is continuous $[0,T^*]$, we can apply the Gronwall lemma
to conclude that $(u_i)_-(t)=0$ and hence $u_i(t)\ge 0$ for $t\in[0,T^*]$.

Now, we show the conservation of mass. 

\begin{lemma}[Conservation of mass]\label{lem.mass}
Let $u=u^{(\eps,\kappa,\rho)}$ be a weak solution to \eqref{3.eq}--\eqref{3.ic}
on $[0,T^*]$. Then $\|u_i(t)\|_1 = \|u_i^0\|_1$ for any $t\in[0,T^*]$.
\end{lemma}

\begin{proof}
Let $R\ge 1$, $\gamma>d$ and introduce the
cutoff function $\psi_R:\R^d\to[0,\infty)$ by
$$
  \psi_R(x) = \psi_1(x/R), \quad \psi_1(x) = (1+|x|^2)^{-\gamma/2}
	\quad\mbox{for }x\in\R^d.
$$
The following estimates hold:
\begin{equation}\label{3.psi1}
  |\na\psi_R(x)|\le CR^{-1}\psi_R(x), \quad |\Delta\psi_R(x)|\le CR^{-2}\psi_R(x)
	\quad\mbox{for }x\in\R^d.
\end{equation}
We claim that
\begin{equation}\label{3.psiR}
  -(-\Delta)^\alpha\psi_R(x)\le CR^{-2\alpha}\psi_R(x)\mbox{ for }x\in\R^d, \quad
	\lim_{R\to\infty}\|(-\Delta)^\alpha\psi_R\|_\infty = 0.
\end{equation}
It is sufficient to prove the first statement for $R=1$, 
thanks to a scaling argument, while
the proof for $R=1$ is similar to that one for \eqref{2.xm}.
%By \cite[Lemma 3.2]{NPV12}, 
%\begin{align*}
%  & -(-\Delta)^\alpha\psi_1(x) = c_{d,\alpha}\int_{\R^d}
%	\frac{\psi_1(x+y)+\psi_1(x-y)-2\psi_1(x)}{|x-y|^{d+2\alpha}}dy =: I_4+I_5,
%	\quad\mbox{where} \\
%  & I_4 = c_{d,\alpha}\int_{\{|y|<1\}}
%	\frac{\psi_1(x+y)+\psi_1(x-y)-2\psi_1(x)}{|x-y|^{d+2\alpha}}dy, \\
%	& I_5 = c_{d,\alpha}\int_{\{|y|\ge 1\}}
%	\frac{\psi_1(x+y)+\psi_1(x-y)-2\psi_1(x)}{|x-y|^{d+2\alpha}}dy.
%\end{align*}
%It follows from Taylor's formula that, for a suitable $\xi=\xi(x,y)\in[0,1]$,
%$$
%  I_4 = c_{d,\alpha}\int_{B_1(0)}\frac{y^T(D^2\psi_1(x+\xi y)+D^2\psi_1(x-\xi y))y}{
%	2|y|^{d+2\alpha}}dy.
%$$
%Estimating for $|y|<1$ according to
%\begin{align*}
%  y^TD^2\psi_1(x+\xi y)y &\le C(1+|x-\xi y|^2)^{-m/2}|y|^2 \\
%	&\le C(1+|x|^2/2-|y|^2)^{-m/2}|y|^2 \le C|x|^{-m}|y|^2, \\
%	y^TD^2\psi_1(x+\xi y)y &\le C(1+|x-\xi y|^2)^{-m/2}|y|^2 \le C|y|^2,
%\end{align*}
%we infer that $y^TD^2\psi_1(x+\xi y)y\le C\psi_1(x)|y|^2$ for $|y|<1$.
%This shows that
%$$
%  I_4 \le C\psi_1(x)\int_{B_1(0)}|y|^{2-d-2\alpha}dy \le C\psi_1(x).
%$$
%The estimate for $I_5$ is easier:
%$$
%  I_5 \le c_{d,\alpha}\int_{\R^d\setminus B_1(0)}\frac{2\psi_1(x)}{|y|^{d+2\alpha}}dy
%	\le C\psi_1(x).
%$$
%This proves the first statement in \eqref{3.psiR}. 
The second statement in \eqref{3.psiR} follows from
$(-\Delta)^\alpha\psi_R(x)=R^{-2\alpha}((-\Delta)^\alpha\psi_1)(x/R)$ and
the property $(-\Delta)^\alpha\psi_1\in L^\infty(\R^d)$.

Since $\psi_R\in H^1(\R^d)$ for $\gamma>d$, we can use $\psi_R$ as a test function in
the weak formulation of \eqref{3.eq}:
\begin{align}\label{3.aux}
  \int_{\R^d}&u_i(t)\psi_R dx - \int_{\R^d}u_i^0\psi_R dx
	= \kappa\int_0^t\int_{\R^d}u_i\Delta\psi_R dxds
	- \kappa\int_0^t\int_{\R^d}g_\rho[u_i]\psi_R dxds \\
	&{}- \sigma_i\int_0^t\int_{\R^d}u_i(-\Delta)^\alpha\psi_R dxds
	- \sum_{j=1}^n a_{ij}\int_0^t\int_{\R^d}u_i\na\psi_R
	\cdot\na\K_{(1-\beta)/2}^{(\eps)}*u_j dxds. \nonumber
\end{align}
We deduce from \eqref{3.psi1} that
\begin{align*}
  \int_{\R^d}&u_i(t)\psi_R dx - \int_{\R^d}u_i^0\psi_R dx
	\le 2\kappa\int_0^t\|u_i\|_2^2 ds + CR^{-2}\int_0^t\int_{\R^d}u_i\psi_R dxds \\
	&\phantom{xx}{}+ CR^{-1}\sum_{j=1}^n\int_0^t\|u_i\|_2
	\|\na\K_{(1-\beta)/2}^{(\eps)}*u_j\|_2ds \\
	&\le 2\kappa\int_0^t\|u_i\|_2^2 ds + C\int_0^t\int_{\R^d}u_i\psi_R dxds
	+ C\sum_{j=1}^n\int_0^t\|u_i\|_2\|u_j\|_2ds \\
	&\le C\sum_{j=1}^n\int_0^t\|u_j\|_2^2 ds + C\int_0^t\int_{\R^d}u_i\psi_R dxds.
\end{align*}
Summing this inequality over $i=1,\ldots,n$, observing that
$u_i\in C^0([0,T^*];L^2(\R^d))$, and applying Gronwall's lemma
shows that
$$
  \sup_{0<t<T^*}\int_{\R^d}u_i(t)\psi_R dx \le C(T^*).
$$
The monotone convergence theorem allows us to perform the limit $R\to\infty$
leading to
$$
  \sup_{0<t<T^*}\int_{\R^d}u_i(t) dx \le C(T^*).
$$
At this point, because of \eqref{3.psi1}, \eqref{3.psiR}, and the fact that
$\int_{\R^d}g_\rho[u_i]dx=0$, the limit $R\to\infty$ in \eqref{3.aux} gives
the conservation of mass:
$$
  \int_{\R^d}u_i(t)dx - \int_{\R^d}u_i^0 dx = 0\quad\mbox{for }t\in[0,T^*],
$$
finishing the proof.
\end{proof}

The next step is the proof of a bound for $u_i$ in $C^0([0,T^*];L^2(\R^d))$,
which allows us to extend the local solution globally.

\begin{lemma}[$L^2(\R^d)$ estimate]\label{lem.L2}
Let $u=u^{(\eps,\kappa,\rho)}$ be a weak solution to \eqref{3.eq}--\eqref{3.ic}
on $[0,T^*]$. Then 
$$
  \|u_i\|_{L^\infty(0,T^*;L^2(\R^d))} 
	+ \sqrt{\kappa}\|\na u_i\|_{L^2(0,T^*;L^2(\R^d))}
	\le C(\eps,T^*).
$$
\end{lemma}

\begin{proof}
We use $u_i$ as a test function in \eqref{3.eq} and estimate
in a similar way as before:
\begin{align*}
  \frac12\|&u_i\|_2^2 - \frac12\|u_i^0\|_2^2
	+ \kappa\int_0^t\|\na u_i\|_2^2 ds 
	+ \sigma_i\int_0^t\|(-\Delta)^{\alpha/2}u_i\|_2^2ds \\
	&= -\kappa\int_0^t\int_{\R^d}g_\rho[u_i]u_i dxds
	- \sum_{j=1}^n a_{ij}\int_0^t\int_{\R^d}u_i\na u_i
	\cdot\na\K_{(1-\beta)/2}^{(\eps)}*u_j dxds \\
	&\le -\kappa\int_0^t\int_{\R^d}g_\rho[u_i]u_i dxds
	+ \frac12\sum_{j=1}^n a_{ij}\int_0^t\int_{\R^d}u_i^2\Delta\K_{(1-\beta)/2}^{(\eps)}
	*u_j dxds \\
	&\le C(\eps)\sum_{j=1}^n\|u_j\|_{L^\infty(0,T;L^1(\R^d))}
	\int_0^t\|u_i\|_2^2 ds.
\end{align*}
Then mass conservation and Gronwall's lemma yield the conclusion.
\end{proof}

We deduce from Lemma \ref{lem.L2} that the solution $u$ to \eqref{3.eq}--\eqref{3.ic}
exists for all $t\ge 0$.

%%%%%%%%%%%%%%%%%%%%%%%%%%%%%%%%%%%%%%%%%%%%%%%%%%%%%%%%%%%%%%%%%%%%%%%%%%%%%%%

\section{Limit in the approximate problem}\label{sec.limit}

We first derive some estimates uniform in $(\eps,\kappa,\rho)$ and perform then
the limits $\rho\to 0$, $\eps\to 0$, and $\kappa\to 0$ in this order.

\subsection{Uniform estimates}

A uniform bound for a moment of $u_i=u_i^{(\eps,\kappa,\rho)}$ can be derived
in a similar way as in Lemma \ref{lem.moment}. 
To make the proof rigorous, we may proceed as in the
proof of the conservation of mass in Section \ref{sec.global} by testing
\eqref{3.eq} with $(1+|\cdot|^2)^{m/2}\psi_R$. This leads to the estimate
\begin{equation}\label{4.moment}
  \sup_{0<t<T}\int_{\R^d}(1+|x|^{2})^{m/2}u_i(t)dx \le C(\eps,u^0,T), \quad\mbox{where }
	0<m<\min\{1,2\alpha\}.
\end{equation}

The following lemma states the entropy inequality for the approximate problem.

\begin{lemma}[Entropy inequality for the approximate problem]\label{lem.ei}
Let $u=u^{(\eps,\kappa,\rho)}$ be a weak solution to \eqref{3.eq}--\eqref{3.ic}.
Then there exists a constant $C>0$ that is independent of $(\eps,\kappa,\rho)$
such that for $t>0$,
\begin{align}\label{4.ei}
  \sum_{i=1}^n&\pi_i\int_{\R^d}u_i(t)\log u_i(t)dx
	+ 4\kappa\sum_{i=1}^n\pi_i\int_0^t\int_{\R^d}|\na\sqrt{u_i}|^2 dxds \\
	&\phantom{xx}{}
	+ C\sum_{i=1}^n\sigma_i\int_0^t\int_{\R^d}|(-\Delta)^{\alpha/2}\sqrt{u_i}|^2dxds
	+ \lambda\sum_{i=1}^n\int_0^t\int_{\R^d}|\na\widetilde{\K}_{(1-\beta)/4}^{(\eps)}
	*u_i|^2 dxds \nonumber \\
	&\phantom{xx}{}
	+ \kappa\sum_{i=1}^n\pi_i\int_0^t\int_{\R^d}u_i(\log u_i)_+W_\rho*u_i dxds 
	\nonumber \\
	&\le \sum_{i=1}^n\pi_i\int_{\R^d}u_i^0\log u_i^0 dx + \kappa Ct
	+ \kappa C\int_0^t\int_{\R^d}u_i^2dxds, \nonumber 
\end{align}
recalling that $\lambda>0$ is the smallest eigenvalue of 
$(\pi_ia_{ij})\in\R^{n\times n}$.
\end{lemma}

\begin{proof}
The usual idea to derive the entropy estimate is to use 
$\pi_i\log u_i$ as a test function
in the weak formulation of \eqref{3.eq}. Since this function is not an element of
$L^2(0,T;H^1(\R^d))$, we need to regularize. 
Instead, we use $\pi_i(\log(u_i+\eta)-\log\eta)
\in L^2(0,T;H^1(\R^d))$ with $0<\eta<1$ as a test function. Thanks to mass
conservation, we have
\begin{align*}
  \langle\pa_t u_i,\log(u_i+\eta)-\log\eta\rangle
	&= \frac{d}{dt}\int_{\R^d}\big((u_i+\eta)\log(u_i+\eta) - \eta\log\eta
	- (1+\log\eta)u_i\big)dx \\
	&= \frac{d}{dt}\int_{\R^d}\big((u_i+\eta)\log(u_i+\eta) - \eta\log\eta\big)dx.
\end{align*}
Setting $H_\eta[u] = \sum_{i=1}^n\pi_i\int_{\R^d}
((u_i+\eta)\log(u_i+\eta)-\eta\log\eta)dx$, we infer from the weak formulation
of \eqref{3.eq}, after summing over $i=1,\ldots,n$, that
\begin{align*}
  & H_\eta[u(t)] - H_\eta[u^0]
	+ 4\kappa\sum_{i=1}^n\pi_i\int_0^t\int_{\R^d}|\na\sqrt{u_i+\eta}|^2dxds
	=: I_6 + I_7 + I_8, \quad\mbox{where} \\
	& I_6 = -\sum_{i=1}^n\sigma_i\pi_i\int_0^t\int_{\R^d}\log(u_i+\eta)(-\Delta)^\alpha
	u_i dxds, \\
	& I_7 = -\kappa\sum_{i=1}^n\pi_i\int_0^t\int_{\R^d}g_\rho[u_i]\log(u_i+\eta)dxds, \\
	& I_8 = -\sum_{i,j=1}^n\pi_ia_{ij}\int_0^t\int_{\R^d}\frac{u_i}{u_i+\eta}
	\na u_i\cdot\na\K_{(1-\beta)/2}^{(\eps)}*u_j dxds.
\end{align*}
We use the generalized Stroock--Varopoulos inequality (Lemma \ref{lem.GSVI})
to estimate $I_6$:
$$
  I_6 \le -C\sum_{i=1}^n\int_0^t\int_{\R^d}|(-\Delta)^{\alpha/2}\sqrt{u_i+\eta}|^2 dxds.
$$
The definition of $g_\rho[u_i]$ yields
\begin{align*}
  I_7 &= -\kappa\sum_{i=1}^n\pi_i\int_0^t\int_{\R^d}u_i\log(u_i+\eta)W_\rho*u_i dxds \\
	&\phantom{xx}{}+ \kappa\sum_{i=1}^n\pi_i\int_0^t\bigg(\int_{\R^d}\log(u_i+\eta)
	\frac{e^{-|x|^2}}{\pi^{d/2}}dx\bigg)\bigg(\int_{\R^d}u_iW_\rho*u_i dy\bigg)ds \\
	&=: I_{71} + I_{72}.
\end{align*}
Since the logarithm is increasing, we find that
\begin{align*}
  I_{71} &\le -\kappa\sum_{i=1}^n\pi_i\int_0^t\int_{\R^d}u_i\log u_i W_\rho*u_i dxds \\
	&\le -\kappa\sum_{i=1}^n\pi_i\int_0^t\int_{\R^d}u_i(\log u_i)_+ W_\rho*u_i dxds
	+ \kappa\sum_{i=1}^n\pi_i\int_0^t\|u_i(\log u_i)_-\|_2\|u_i\|_2 ds \\
	&\le -\kappa\sum_{i=1}^n\pi_i\int_0^t\int_{\R^d}u_i(\log u_i)_+ W_\rho*u_i dxds
	+ \kappa C\sum_{i=1}^n\int_0^t\|u_i\|_{1}^{1/2}\|u_i\|_2 ds,
\end{align*}
where we used the inequality $u_i^2(\log u_i)_-^{2}\le u_i$ in the last step.
The inequality $\log(u_i+\eta)\le C(1+u_i)$ and mass conservation imply that
$$
  I_{72} \le \kappa \sum_{i=1}^n\int_0^t\bigg(C + C\int_{\R^d}u_i
	\frac{e^{-|x|^2}}{\pi^{d/2}}dx\bigg)\bigg(\int_{\R^d}u_iW_\rho*u_i dy\bigg)ds
  \le \kappa C\int_0^t\|u_i\|_2^2ds.
$$
We infer that
$$
  I_7 \le \kappa Ct + \kappa C\int_0^t\|u_i\|_2^2ds
	-\kappa\sum_{i=1}^n\pi_i\int_0^t\int_{\R^d}u_i(\log u_i)_+ W_\rho*u_i dxds.
$$
Finally, by the definition of $\K_{(1-\beta)/2}^{(\eps)}$, the positive definiteness
of the matrix $(\pi_ia_{ij})$, and integration by parts,
\begin{align*}
  I_8 &= -\sum_{i,j=1}^n\pi_i a_{ij}\int_0^t\int_{\R^d}
	(\na\widetilde{\K}_{(1-\beta)/4}^{(\eps)}*u_i)
	\cdot(\na\widetilde{\K}_{(1-\beta)/4}^{(\eps)}*u_j)dxds \\
	&\phantom{xx}{}+ \sum_{i,j=1}^n\pi_i a_{ij}\int_0^t\int_{\R^d}
	\frac{\eta}{u_i+\eta}\na u_i\cdot\na\K_{(1-\beta)/2}^{(\eps)}*u_jdxds \\
	&\le -\lambda\sum_{i=1}^n\int_0^t\int_{\R^d}
	|\na\widetilde{\K}_{(1-\beta)/4}^{(\eps)}*u_i|^2 dxds + I_{81}(\eta), 
\end{align*}
where 
$$
  I_{81}(\eta) = - \sum_{i,j=1}^n\pi_ia_{ij}\int_0^t\int_{\R^d}
	\eta(\log(u_i+\eta)-\log\eta)\Delta\K_{(1-\beta)/2}^{(\eps)}*u_j dxds.
$$
We summarize the previous estimates:
\begin{align}
  H_\eta&[u(t)] - H_\eta[u^0] + 4\kappa\sum_{i=1}^n\pi_i\int_0^t\int_{\R^d}
	|\na\sqrt{u_i+\eta}|^2 dxds \nonumber \\
	&\phantom{xx}{}+ C\sum_{i=1}^n\int_0^t\int_{\R^d}|(-\Delta)^{\alpha/2}
	\sqrt{u_i+\eta}|^2 dxds 
	+ \kappa\sum_{i=1}^n\pi_i\int_0^t\int_{\R^d}u_i(\log u_i)_+W_\rho*u_i dxds 
	\label{4.aux2} \\
	&\phantom{xx}{}+ \lambda\sum_{i=1}^n\int_0^t\int_{\R^d}
	|\na\widetilde{\K}_{(1-\beta)/4}^{(\eps)}*u_i|^2 dxds
	\le \kappa Ct + \kappa C\int_0^t\|u_i\|_2^2 ds + I_{81}(\eta). \nonumber
\end{align}
Before performing the limit $\eta\to 0$, we estimate the
error term $I_{81}(\eta)$:
\begin{align*}
  I_{81}(\eta) &\le C\sum_{i,j=1}^n\|\Delta\K_{(1-\beta)/2}^{(\eps)}*u_j
	\|_{L^\infty(0,T;L^{\infty}(\R^d))}\int_0^t\int_{\R^d}
	\eta(\log(u_i+\eta)-\log\eta)dxds \\
	&\le C(\eps)\sum_{i,j=1}^n\|u_j\|_{L^\infty(0,T;L^1(\R^d))}
	\int_0^t\int_{\R^d}\eta(\log(u_i+\eta)-\log\eta)dxds.
\end{align*}
By mass conservation, the first factor is bounded, while the second one tends to
zero as $\eps\to 0$. Indeed, it holds that $\eta(\log(u_i+\eta)-\log\eta)\to 0$
a.e.\ in $\R^d\times(0,T)$ as $\eta\to 0$ and $0\le \eta(\log(u_i+\eta)-\log\eta)
\le u_i\in L^\infty(0,T;L^1(\R^d))$, and therefore, we can apply the dominated
convergence theorem leading to $I_{81}(\eta)\to 0$ as $\eta\to 0$.

At this point, we can take the limit $\eta\to 0$ in \eqref{4.aux2} by applying
dominated convergence, Fatou's lemma, and the weak lower semicontinuity of the
$L^2(\R^d)$ norm to conclude the proof.
\end{proof}

We deduce from the upper bound for $u_i\log u_i$, mass conservation, and the
moment bound that $u_i\log u_i$ is bounded in $L^1(\R^d)$, as stated in the
following lemma.

\begin{lemma}\label{lem.ulogu}
Let $u=u^{(\eps,\kappa,\rho)}$ be a weak solution to \eqref{3.eq}--\eqref{3.ic}.
Then for any $T>0$,
$$
  \|u_i\log u_i\|_{L^\infty(0,T;L^1(\R^d))} \le C.
$$
\end{lemma}

\begin{proof}
The proof is similar to that one in \cite[Section 2]{GuZa18}. 
In fact, the result holds for
any function $0\le v\in L^\infty(0,T;L^1(\R^d))$ satisfying
$$
  \sup_{0<t<T}\int_{\R^d} v(t)\big(\log v(t)+(1+|x|^2)^{m/2}\big)dx \le C(T),
$$
where $m>0$. We show that $\sup_{0<t<T}\|v(t)\log v(t)\|_1\le C(T)$. For this, we write
\begin{align*}
  \int_{\R^d}|v\log v|dx &= -\int_{\{v<1\}}v\log v dx + \int_{\{v\ge 1\}}v\log v dx \\
	&= -2\int_{\{v<1\}}v\log v dx + \int_{\R^d}v\log vdx
	\le -2\int_{\{v<1\}}v\log v dx + C.
\end{align*}
We use the Cauchy--Schwarz inequality to estimate the integral on the right-hand side:
\begin{align*}
  -\int_{\{v<1\}}&v\log v dx = \int_{\{v<1\}}v^{(1-\delta)/2}v^{(1+\delta)/2}
	\log\frac{1}{v}dx \\
	&\le \bigg(\int_{\{v<1\}}v^{1-\delta}dx\bigg)^{1/2}\bigg(\int_{\{v<1\}}
	v\bigg(v^{\delta/2}\log\frac{1}{v}\bigg)^2 dx\bigg)^{1/2},
\end{align*}
where $\delta\in(0,1)$.
The function $(0,1)\to\R$, $s\mapsto s^{\delta/2}\log(1/s)$, is bounded by a
constant $C(\delta)$. Therefore, taking into account mass conservation for $v$
and the H\"older inequality,
\begin{align*}
  -\int_{\{v<1\}}&v\log v dx \le C(\delta)\bigg(\int_{\{v<1\}}
	v^{1-\delta}dx\bigg)^{1/2} \\
	&= C(\delta)\bigg(\int_{\{v<1\}}(1+|x|^2)^{m(1-\delta)/2}v(x)^{1-\delta}
	(1+|x|^2)^{-m(1-\delta)/2}dx\bigg)^{1/2} \\
	&\le C(\delta)\bigg(\int_{\{v<1\}}(1+|x|^2)^{m/2}v(x)dx\bigg)^{(1-\delta)/2}
	\bigg(\int_{\{v<1\}}(1+|x|^2)^{-m(1-\delta)/(2\delta)}dx\bigg)^{\delta/2}.
\end{align*}
The moment estimate shows that the first integral is bounded, while the second one
is finite if $m(1-\delta)/(2\delta)>d$ or $\delta<m/(m+2d)$. This proves the claim.
\end{proof}

We deduce from the previous lemmas the following estimates.

\begin{lemma}[Uniform estimates]\label{lem.est}
Let $u=u^{(\eps,\kappa,\rho)}$ be a weak solution to \eqref{3.eq}--\eqref{3.ic}.
Then there exist constants $q>1$ and $C(\eps,T)>0$, which is independent of 
$(\kappa,\rho)$, such that for $t>0$,
\begin{align}
  \sqrt{\kappa}\|\sqrt{u_i}\|_{L^2(0,T;H^1(\R^d))}
	+ \sqrt{\sigma_i}\|\sqrt{u_i}\|_{L^2(0,T;H^\alpha(\R^d))} &\le C(\eps,T), 
	\label{est1} \\
	\|\na\widetilde{\K}_{(1-\beta)/4}^{(\eps)}*u_i\|_{L^2(0,T;L^2(\R^d))}
	+ \kappa\|u_i(\log u_i)_+W_\rho*u_i\|_{L^1(0,T;L^1(\R^d))}
	&\le C(\eps,T), \label{est2} \\
	\|u_i\log u_i\|_{L^\infty(0,T;L^1(\R^d))} 
	+ \|\pa_t u_i\|_{L^q(0,T;W^{-1,q}(\R^d))} &\le C(\eps,T). \label{est3}
\end{align}
\end{lemma}

\begin{proof}
Estimates \eqref{est1} and \eqref{est2} follow from Lemmas \ref{lem.mass}
and \ref{lem.ei}. The first estimate in \eqref{est3}
is proved in Lemma \ref{lem.ulogu}. 
It remains to prove the second estimate in \eqref{est3}. 

Let $p>\max\{d/(1-\alpha),2d/(1-\beta)\}>2$ with $1/p+1/q=1$ and use 
$\phi\in C^0([0,T];C_0^\infty(\R^d))$
as a test function in the weak formulation of \eqref{3.eq}: 
\begin{align}
  & \int_0^T\langle\pa_t u_i,\phi\rangle dt =: I_9 + \cdots + I_{12}, 
	\quad\mbox{where } \label{4.aux3} \\
	& I_9 = -\kappa\int_0^T\int_{\R^d}\na u_i\cdot\na\phi dxdt, \nonumber \\ 
	& I_{10} = -\sigma_i\int_0^T\int_{\R^d}(-\Delta)^{\alpha/2}u_i
	(-\Delta)^{\alpha/2}\phi dxdt, \nonumber \\
	& I_{11} = -\kappa\int_0^T\int_{\R^d}g_\rho[u_i]\phi dxdt, \nonumber \\
	& I_{12} = -\sum_{j=1}^na_{ij}\int_0^T\int_{\R^d}u_i\na\phi\cdot\big(
	\widetilde{\K}_{(1-\beta)/4}^{(\eps)}*(\na\widetilde{\K}_{(1-\beta)/4}^{(\eps)}*u_j)
	\big)dxdt. \nonumber 
\end{align}
We estimate the integrals $I_9,\ldots,I_{12}$. First, by Lemma \ref{lem.L2} 
with $T^*=T$, it holds that $\sqrt{\kappa}\|\na u_i\|_{L^2(0,T; L^2(\R^d))}
\leq C(\eps,T)$.
We infer from \eqref{est1} that $\sqrt{\kappa}\na u_i=2\sqrt{\kappa}
\sqrt{u_i}\na\sqrt{u_i}$ is bounded in $L^1(\R^d)$, i.e.\ 
$\sqrt{\kappa}\|\na u_i\|_{L^1(0,T; L^1(\R^d))}\leq C(\eps,T)$. 
Hence, since $q<2$, it follows by interpolation that 
$\sqrt{\kappa}\|\na u_i\|_{L^q(0,T; L^q(\R^d))}\leq C(\eps,T)$. We deduce that
$$
  |I_9| \le \kappa\|\na u_i\|_{L^q(0,T;L^q(\R^d))}\|\na\phi\|_{L^p(0,T;L^p(\R^d))}
	\le C\|\phi\|_{L^p(0,T;W^{1,p}(\R^d))}.
$$
We can prove, using the generalized Stroock--Varopoulos inequality 
(Lemma \ref{lem.GSVI}) in a similar way as in Lemma \ref{lem.beta}, that
$\sqrt{\sigma_i}\|(-\Delta)^{\alpha/2}u_i\|_{L^2(0,T;L^1(\R^d))} \le C$. Therefore,
since $p>d/(1-\alpha)$,
$$
  |I_{10}| \le \sigma_i\|(-\Delta)^{\alpha/2}u_i\|_{L^2(0,T;L^1(\R^d))}
	\|(-\Delta)^{\alpha/2}\phi\|_{L^2(0,T;L^\infty(\R^d))}
	\le C\|\phi\|_{L^2(0,T;W^{1,p}(\R^d))}.
$$
It follows from property \eqref{3.g1} of $g_\rho[u_i]$, the $L^\infty(0,T;L^2(\R^d))$
estimate of $u_i$ in Lemma \ref{lem.L2}, and the embedding $W^{1,p}(\R^d)
\hookrightarrow L^\infty(\R^d)$ that
$$
  |I_{11}| \le \int_0^T\|g_\rho[u_i]\|_1\|\phi\|_\infty dt
	\le C\int_0^T\|u_i\|_2^2\|\phi\|_{W^{1,p}(\R^d)} dt
	\le C(\eps,T)\|\phi\|_{L^p(0,T;W^{1,p}(\R^d))}.
$$
Finally, the Hardy--Littlewood--Sobolev inequality
(Lemma \ref{lem.HLS} with $r=2d/(d+1-\beta)$) and the H\"older inequality 
with $1/q_2 + 1/p = 1/r$ lead to
\begin{align*}
  |I_{12}| &\le C(\eps)\int_0^T
	\|u_i\na\phi\|_r\|\na\widetilde{\K}_{(1-\beta)/4}^{(\eps)}*u_j\|_2 dt \\
	&\le C(\eps)\int_0^T\|u_i\|_{q_2}\|\na\phi\|_{p}
	\|\na\widetilde{\K}_{(1-\beta)/4}^{(\eps)}*u_j\|_2 dt 
	\le C(\eps,T)\|\na\phi\|_{L^{p}(0,T;L^{p}(\R^d))},
\end{align*}
where we used Lemma \ref{lem.L2}, mass conservation, the fact that $q_2\in [1,2)$, 
and estimate \eqref{est2} in the last step. 
Putting together the estimates for $I_9,\ldots,I_{12}$,
we conclude the proof from \eqref{4.aux3} for $\phi\in L^p(0,T;W^{1,p}(\R^d))$
with $p>\max\{d/(1-\alpha),2d/(1-\beta)\}$.
\end{proof}

\subsection{Limit $\rho\to 0$}

We conclude from Lemma \ref{lem.comp} in the Appendix that
$$
  V:=\bigg\{v\in H^1(\R^d):
	\int_{\R^d}(1+|x|^2)^{m/2}|v(x)| dx<\infty\bigg\}
$$
is compactly embedded into $L^2(\R^d)$.
Moreover, the embedding $L^2(\R^d)\hookrightarrow H^{-s}(\R^d)$
is continuous for any $s>0$.
The uniform $L^2(0,T;H^1(\R^d))$ bound in Lemma 
\ref{lem.L2} and the moment bound \eqref{4.moment} 
show that $(u^{(\eps,\kappa,\rho)})$ is bounded in $L^2(0,T;V)$, while,
by estimate \eqref{est3}, $(\pa_t u_i^{(\eps,\kappa,\rho)})$ is bounded in
$L^1(0,T;H^{-s}(\R^d))$. It follows from the Aubin--Lions lemma
that there exists a subsequence, which is not relabeled, such that, as $\rho\to 0$,
$$
  u_i^{(\eps,\kappa,\rho)} \to u_i \quad\mbox{strongly in }L^2(0,T;L^2(\R^d)).
$$
Since $(u_i^{(\eps,\kappa,\rho)})$ is bounded in $L^\infty(0,T;L^2(\R^d))
\cap L^2(0,T;H^1(\R^d))$ by Lemma \ref{lem.L2}, the Gagliar\-do--Nirenberg
inequality provides a uniform bound in $L^{2+4/d}(0,T;L^{2+4/d}(\R^d))$.
Hence, there exists $2<p<2+4/d$ such that
\begin{equation}\label{4.Lp}
  u_i^{(\eps,\kappa,\rho)} \to u_i \quad\mbox{strongly in }L^p(0,T;L^p(\R^d)).
\end{equation}
Given the uniform bounds in Lemma \ref{lem.est}, it is quite standard to perform
the limit $\rho\to 0$ in \eqref{3.eq}. We consider here only
the term that explicitly depends on $\rho$, namely
\begin{align*}
  \int_0^T\int_{\R^d}&g_\rho[u_i^{(\eps,\kappa,\rho)}]\phi dxdt
	= \int_0^T\int_{\R^d}u_i^{(\eps,\kappa,\rho)}(W_\rho*u_i^{(\eps,\kappa,\rho)})
	\phi dxdt \\
	&{}- \int_0^T\bigg(\int_{\R^d}u_i^{(\eps,\kappa,\rho)}
	(W_\rho*u_i^{(\eps,\kappa,\rho)})dx\bigg)\bigg(\int_{\R^d}\frac{e^{-|x|^2}}{\pi^{d/2}}
	\phi dx\bigg)dt
\end{align*}
for test functions $\phi\in L^2(0,T;L^\infty(\R^d))$.
Since $\|W_\rho\|_1=1$ and $(u_i^{(\eps,\kappa,\rho)})$ is bounded in
$L^\infty(0,T;L^2(\R^d))$, we have
$$
  W_\rho*u_i^{(\eps,\kappa,\rho)}\rightharpoonup^* u_i
	\quad\mbox{weakly* in }L^\infty(0,T;L^2(\R^d)).
$$
This implies that, for suitable test functions,
$$
  \int_0^T\int_{\R^d}g_\rho[u_i^{(\eps,\kappa,\rho)}]\phi dxdt
	\to \int_0^T\int_{\R^d}g_0[u_i]\phi dxdt\quad\mbox{as }\eps\to 0,
$$
where 
$$
  g_0[v](x) := v(x)^2 - \frac{e^{-|x|^2}}{\pi^{d/2}}\int_{\R^d}v(y)^2dy,
	\quad v\in L^2(\R^d).
$$
We have proved that the limit $u_i^{(\eps,\kappa)}:=u_i$ is a solution to
\begin{align}\nonumber
  & \pa_t u_i^{(\eps,\kappa)} - \kappa\Delta u_i^{(\eps,\kappa)}
	+ \sigma_i(-\Delta)^\alpha u_i^{(\eps,\kappa)} + \kappa g_0[u_i^{(\eps,\kappa)}] \\
	&\phantom{xx}{}
	= \diver\bigg(\sum_{j=1}^n a_{ij} u_i^{(\eps,\kappa)}\na\K_{(1-\beta)/2}^{(\eps)}
	* u_j^{(\eps,\kappa)}\bigg)\quad\mbox{in }\R^d,\ t>0, \label{4.uepskappa} \\
	& u_i^{(\eps,\kappa)}(\cdot,0) = u_i^0\quad\mbox{in }\R^d,\ i=1,\ldots,n.
	\nonumber
\end{align}

The strong convergence \eqref{4.Lp}, Fatou's lemma, and the weak lower
semicontinuity of the $L^2(\R^d)$ norm allow us to take the limit $\rho\to 0$
in the approximate entropy inequality \eqref{4.ei}, leading to
\begin{align*}
  \sum_{i=1}^n&\pi_i\int_{\R^d}u_i^{(\eps,\kappa)}(t)\log u_i^{(\eps,\kappa)}(t)dx
	+ 4\kappa\sum_{i=1}^n\pi_i\int_0^t\int_{\R^d}
	\Big|\na\sqrt{u_i^{(\eps,\kappa)}}\Big|^2 dxds \\
	&\phantom{xx}{}
	+ C\sum_{i=1}^n\sigma_i\int_0^t\int_{\R^d}\Big|(-\Delta)^{\alpha/2}
	\sqrt{u_i^{(\eps,\kappa)}}\Big|^2dxds
	+ \lambda\sum_{i=1}^n\int_0^t\int_{\R^d}|\na\widetilde{\K}_{(1-\beta)/4}^{(\eps)}
	* u_i^{(\eps,\kappa)}|^2 dxds \\
	&\phantom{xx}{}
	+ \kappa\sum_{i=1}^n\pi_i\int_0^t\int_{\R^d}
	(u_i^{(\eps,\kappa)})^2(\log u_i^{(\eps,\kappa)})_+ dxds \\
	&\le \sum_{i=1}^n\pi_i\int_{\R^d}u_i^0\log u_i^0 dx + \kappa Ct
	+ \kappa C\int_0^t\int_{\R^d}(u_i^{(\eps,\kappa)})^2dxds. 
\end{align*}
The last integral on the right-hand side can be controlled by the last integral
on the left-hand side. Therefore,
\begin{align}
  \sum_{i=1}^n&\pi_i\int_{\R^d}u_i^{(\eps,\kappa)}(t)\log u_i^{(\eps,\kappa)}(t)dx
	+ 4\kappa\sum_{i=1}^n\pi_i\int_0^t\int_{\R^d}
	\Big|\na\sqrt{u_i^{(\eps,\kappa)}}\Big|^2 dxds \nonumber \\
	&{}+ C\sum_{i=1}^n\sigma_i\int_0^t\int_{\R^d}\Big|(-\Delta)^{\alpha/2}
	\sqrt{u_i^{(\eps,\kappa)}}\Big|^2dxds
	+ \lambda\sum_{i=1}^n\int_0^t\int_{\R^d}|\na\widetilde{\K}_{(1-\beta)/4}^{(\eps)}
	u_i^{(\eps,\kappa)}|^2 dxds \label{4.ei2} \\
	&{}+ \kappa\sum_{i=1}^n\pi_i\int_0^t\int_{\R^d}
	(u_i^{(\eps,\kappa)})^2(\log u_i^{(\eps,\kappa)})_+ dxds 
	\le \sum_{i=1}^n\pi_i\int_{\R^d}u_i^0\log u_i^0 dx + \kappa C(t+1). \nonumber 
\end{align}
This shows that the uniform bounds in Lemma \ref{lem.est} also hold for
$u^{(\eps,\kappa)}$ with constants independent of $\eps$.

\begin{lemma}\label{lem.est2}
The solution $u_i:=u_i^{(\eps,\kappa)}$ constructed above satisfies the
following uniform estimates with a constant $C(T)>0$ that is independent of $\eps$
and $\kappa$:
\begin{align}
  \sqrt{\kappa}\|\sqrt{u_i}\|_{L^2(0,T;H^1(\R^d))}
	+ \sqrt{\sigma_i}\|\sqrt{u_i}\|_{L^2(0,T;H^\alpha(\R^d))} &\le C(T), \label{4.ek1} \\
	\|\na\widetilde{\K}_{(1-\beta)/4}^{(\eps)}*u_i\|_{L^2(0,T;L^2(\R^d))}
	+ \kappa\|u_i^2(\log u_i)_+\|_{L^1(0,T;L^1(\R^d))} &\le C(T), \label{4.ek2} \\
	\sqrt{\sigma_i}\|(-\Delta)^{\alpha/2}u_i\|_{L^2(0,T;L^1(\R^d))}
	+ \|\pa_t u_i\|_{L^q(0,T;W^{-1,q}(\R^d))} &\le C(T) \label{4.ek3} \\
	\|u_i\log u_i\|_{L^\infty(0,T;L^1(\R^d))} 
	+ \sup_{0<t<T}\int_{\R^d}|x|^m u_i(t)dx &\le C(T), \label{4.ek4}
\end{align}
where $q>1$.
\end{lemma}

\begin{proof}
Estimates \eqref{4.ek1}--\eqref{4.ek2} follow from \eqref{4.ei2}. 
The first estimate in \eqref{4.ek3} is a consequence of
the $L^\infty(0,T;L^1(\R^d))$ bound for $u_i$ and the $L^2(0,T;L^2(\R^d))$
norm for $(-\Delta)^{\alpha/2}\sqrt{u_i}$; see the proof of \eqref{2.est2}.
The second estimate in \eqref{4.ek3} is shown as in Lemma \ref{lem.est}, now
using the $\eps$-independent entropy estimates.
The moment estimate for $u_i$ can be proved as in Lemma \ref{lem.moment}.
Compared to \eqref{4.moment}, we are able to derive a uniform bound independent
of $\eps$. This is possible since we have an $\eps$-independent 
$L^2(\R^d)$ bound for $u_i$ after having performed the limit $\rho\to 0$. 
This bound is not available for $u_i^{(\eps,\kappa,\rho)}$, since
its $L^2(\R^d)$ estimate depends on $\eps$; see Lemma \ref{lem.L2}.
The critical term becomes, using the Hardy--Littlewood--Sobolev inequality
and a cutoff function $\psi_R$,
\begin{align*}
  \bigg|\sum_{j=1}^n & a_{ij}\int_0^T\int_{\R^d}u_i\na(1+|x|^2)^{m/2}\cdot\big(
	\widetilde{\K}_{(1-\beta)/4}^{(\eps)}
	*(\na\widetilde{\K}_{(1-\beta)/4}^{(\eps)}*u_j)\big)\psi_R dxdt\bigg| \\
  &\le C\int_0^T\|u_i\|_2\|\na(1+|x|^2)^{m/2}\|_{\infty}
	\|\na\widetilde{\K}_{(1-\beta)/4}^{(\eps)}*u_j\|_2 dt \le C.
\end{align*}
Then, proceeding as in the proofs of Lemmas \ref{lem.moment} and \ref{lem.mass}
(to handle the cutoff), we obtain the
moment estimate for $u_i$. This estimate, together with the upper bound for
$\int_{\R^d}u_i\log u_idx$ from \eqref{4.ei2} and the mass conservation property,
imply the $L^\infty(0,T;L^1(\R^d))$ bound for $u_i\log u_i$, by proceeding as 
in the proof of Lemma \ref{lem.ulogu}. 
\end{proof}

\subsection{Limit $\eps\to 0$}

The uniform bounds of Lemma \ref{lem.est2} allow us to apply the compactness
result of Aubin--Lions-type in Lemma \ref{lem.Aubin} below to conclude that
there exists a subsequence (not relabeled) such that
$$
  u_i^{(\eps,\kappa)}\to u_i^{(\kappa)}\quad\mbox{strongly in }L^2(0,T;L^2(\R^d))
	\mbox{ as }\eps\to 0.
$$
We wish to perform the limit $\eps\to 0$ in \eqref{4.uepskappa}. The only
nontrivial term is that one on the right-hand side of \eqref{4.uepskappa}.
We first notice that, by Lemma \ref{lem.est2},
\begin{equation}\label{4.xi}
  \na\widetilde{\K}_{(1-\beta)/4}^{(\eps)}*u_i^{(\eps,\kappa)}\rightharpoonup\xi_i
	\quad\mbox{weakly in }L^2(0,T;L^2(\R^d))
\end{equation}
for some $\xi_i\in L^2(0,T;L^2(\R^d))$, $i=1,\ldots,n$. To identify $\xi_i$, we consider
\begin{equation}\label{4.aux4}
  \int_0^T\int_{\R^d}\phi\big(\widetilde{\K}_{(1-\beta)/4}^{(\eps)}
	*u_i^{(\eps,\kappa)}\big)dxdt
	= \int_0^T\int_{\R^d}u_i^{(\eps,\kappa)}\big(\widetilde{\K}_{(1-\beta)/4}^{(\eps)}
	*\phi\big)dxdt,
\end{equation}
and we wish to pass to the limit $\eps\to 0$ on the right-hand side.
For this, we remark that it follows from the definition of 
$\widetilde{\K}_{(1-\beta)/4}^{(\eps)}$ that 
\begin{align}\label{4.Kphi}
  \|\widetilde{\K}_{(1-\beta)/4}^{(\eps)}*\phi\|_{L^2(0,T;L^{2d/(d-1+\beta)}(\R^d))}
	&\le \|\K_{(1-\beta)/4}*|\phi|\|_{L^2(0,T;L^{2d/(d-1+\beta)}(\R^d))} \\
	&\le C\|\phi\|_{L^2(0,T;L^2(\R^d))}. \nonumber
\end{align}
It holds that $0\le\widetilde{\K}_{(1-\beta)/4}^{(\eps)}\nearrow\K_{(1-\beta)/4}$ 
a.e.\ in $\R^d$ and $\K_{(1-\beta)/4}$ (the kernel of the Riesz potential)
is integrable in the unit ball $B_1(0)$,
while its square $\K_{(1-\beta)/4}^2$ is integrable in $\R^d\setminus B_1(0)$.
Hence, we infer from Young's convolution inequality and the monotone convergence
theorem that
\begin{align*}
  \big\| \big[ \big(&\widetilde{\K}_{(1-\beta)/4}^{(\eps)}
	- \K_{(1-\beta)/4})\mathrm{1}_{B_1(0)} \big]*\phi \big\|_{L^2(0,T;L^2(\R^d))} \\
	&\le \big\|\big(\widetilde{\K}_{(1-\beta)/4}^{(\eps)}
	- \K_{(1-\beta)/4})\mathrm{1}_{B_1(0)}\big\|_1
	\|\phi\|_{L^2(0,T;L^2(\R^d))} \to 0, \\
	\big\| \big[ \big(&\widetilde{\K}_{(1-\beta)/4}^{(\eps)}
	- \K_{(1-\beta)/4})\mathrm{1}_{\R^d\setminus B_1(0)}\big]*\phi 
	\big\|_{L^2(0,T;L^2(\R^d))} \\
	&\le \big\|\big(\widetilde{\K}_{(1-\beta)/4}^{(\eps)}
	- \K_{(1-\beta)/4})\mathrm{1}_{\R^d\setminus B_1(0)}\big\|_2
	\|\phi\|_{L^2(0,T;L^1(\R^d))} \to 0,
\end{align*}
such that for $\phi\in L^2(0,T;L^2(\R^d)\cap L^1(\R^d))$,
\begin{equation}\label{4.convK1}
  \widetilde{\K}_{(1-\beta)/4}^{(\eps)}*\phi\to \K_{(1-\beta)/4}*\phi
	\quad\mbox{strongly in }L^2(0,T;L^2(\R^d)).
\end{equation}
Thus, we deduce from \eqref{4.aux4} and the strong convergence of 
$(u_i^{(\eps,\kappa)})$ in $L^2(0,T;L^2(\R^d))$ that
\begin{align*}
  \int_0^T\int_{\R^d}\phi\big(\widetilde{\K}_{(1-\beta)/4}^{(\eps)}
	*u_i^{(\eps,\kappa)}\big)dxdt
	&\to \int_0^T\int_{\R^d}u_i^{(\kappa)}\K_{(1-\beta)/4}*\phi dxdt \\
	&= \int_0^T\int_{\R^d}\phi\big(\K_{(1-\beta)/4}*u_i^{(\kappa)}\big)dxdt,
\end{align*}
which means that
$$
  \widetilde{\K}_{(1-\beta)/4}^{(\eps)}*u_i^{(\eps,\kappa)}
	\rightharpoonup \K_{(1-\beta)/4}*u_i^{(\kappa)}
	\quad\mbox{weakly in }L^2(0,T;(L^1(\R^d)\cap L^2(\R^d))').
$$
Hence, we can identify the limit $\xi_i$ in \eqref{4.xi}, leading to the
convergence
\begin{equation}\label{4.convK2}
  \na\widetilde{\K}_{(1-\beta)/4}^{(\eps)}*u_i^{(\eps,\kappa)}\rightharpoonup
	\na \K_{(1-\beta)/4}*u_i^{(\kappa)}
	\quad\mbox{weakly in }L^2(0,T;L^2(\R^d)).
\end{equation}

We claim that a similar weak limit holds for $\na\K_{(1-\beta)/2}^{(\eps)}$
instead of $\na\widetilde{\K}_{(1-\beta)/4}^{(\eps)}$. To this end, we use
definition \eqref{2.defK} of $\K_{(1-\beta)/2}^{(\eps)}$ and 
convergences \eqref{4.convK1} and \eqref{4.convK2}:
\begin{align*}
  \int_0^T\int_{\R^d}\phi\na\K_{(1-\beta)/2}^{(\eps)}*u_i^{(\eps,\kappa)}dxdt
	&= \int_0^T\int_{\R^d}\phi\widetilde{\K}_{(1-\beta)/4}^{(\eps)}
	*\na\widetilde{\K}_{(1-\beta)/4}^{(\eps)}*u_i^{(\eps,\kappa)}dxdt \\
	&= \int_0^T\int_{\R^d}\big(\na\widetilde{\K}_{(1-\beta)/4}^{(\eps)}
	*u_i^{(\eps,\kappa)}\big)\big(\widetilde{\K}_{(1-\beta)/4}^{(\eps)}*\phi\big)dxdt \\
	&\to \int_0^T\int_{\R^d}\big(\na\K_{(1-\beta)/4}*u_i^{(\kappa)}\big)
	\big(\K_{(1-\beta)/4}*\phi\big)dxdt \\
	&= \int_0^T\int_{\R^d}\phi\K_{(1-\beta)/4}*\na\K_{(1-\beta)/4}*u_i^{(\kappa)}dxdt \\
	&= \int_0^T\int_{\R^d}\phi\na\K_{(1-\beta)/2}*u_i^{(\kappa)} dxdt
\end{align*}
for any $\phi\in L^2(0,T;L^2(\R^d))$, where we used the
representation $\K_s*f=(-\Delta)^{-s}f$. We infer that
$$
  \na\widetilde{\K}_{(1-\beta)/2}^{(\eps)}*u_i^{(\eps,\kappa)}\rightharpoonup
	\na\K_{(1-\beta)/2}*u_i^{(\kappa)}\quad\mbox{weakly in }L^2(0,T;L^{2}(\R^d)).
$$
Together with the strong $L^2(\R^d)$ convergence of $(u_i^{(\eps,\kappa)})$,
it follows that
$$
  u_i^{(\eps,\kappa)}\na\widetilde{\K}_{(1-\beta)/2}^{(\eps)}*u_j^{(\eps,\kappa)}
	\rightharpoonup u_i^{(\kappa)}\na\K_{(1-\beta)/2}*u_j^{(\kappa)}
	\quad\mbox{weakly in }L^1(0,T;L^1(\R^d)).
$$
These convergences allow us to perform the limit $\eps\to 0$ in
\eqref{4.uepskappa} to conclude that $u^{(\kappa)}$ solves
\begin{align}
  & \pa_t u_i^{(\kappa)} - \kappa\Delta u_i^{(\kappa)} 
	+ \sigma(-\Delta)^\alpha u_i^{(\kappa)} + \kappa g_0[u_i^{(\kappa)}] \nonumber \\
	&\phantom{xx}{}
	= \diver\bigg(\sum_{j=1}^n a_{ij} u_i^{(\kappa)}\na(-\Delta)^{(\beta-1)/2}
	u_j^{(\kappa)}\bigg)\quad\mbox{in }\R^d,\ t>0, \label{4.ukappa} \\
	& u_i^{(\kappa)}(\cdot,0) = u_i^0\quad\mbox{in }\R^d,\ i=1,\ldots,n. \nonumber
\end{align}
The limit $\eps\to 0$ in the entropy inequality \eqref{4.ei2} leads to
\begin{align}
  \sum_{i=1}^n &\pi_i\int_{\R^d}u_i^{(\kappa)}(t)\log u_i^{(\kappa)}(t)dx
	+ 4\kappa\sum_{i=1}^n \pi_i\int_0^t\int_{\R^d}
	\Big|\na\sqrt{u_i^{(\kappa)}}\Big|^2 dxds \nonumber \\
	&{}+ C\sum_{i=1}^n\sigma_i\int_0^t\int_{\R^d}\Big|(-\Delta)^{\alpha/2}
	\sqrt{u_i^{(\kappa)}}\Big|^2 dxds
	+ \lambda\sum_{i=1}^n\int_0^t\int_{\R^d}|\na(-\Delta)^{(\beta-1)/4}u_i^{(\kappa)}|^2
	dxds \label{4.ei3} \\
	&{}+ \kappa\sum_{i=1}^n\pi_i\int_0^t\int_{\R^d}(u_i^{(\kappa)})^2
	(\log u_i^{(\kappa)})_+ dxds
	\le \sum_{i=1}^n\pi_i\int_{\R^d}u_i^0\log u_i^0 dx + \kappa C(t+1) \nonumber
\end{align}
for $t>0$. Estimates \eqref{4.ek1}--\eqref{4.ek4} also hold for $u^{(\kappa)}$
with the exception that the first bound in \eqref{4.ek2} is replaced by
\begin{equation}\label{4.est5}
  \|\na(-\Delta)^{(\beta-1)/4}u_i^{(\kappa)}\|_{L^2(0,T;L^2(\R^d))}\le C,
	\quad i=1,\ldots,n.
\end{equation}

\subsection{Limit $\kappa\to 0$}

We deduce from \eqref{2.Lqp} with $q=2(d+\beta+1)/d$ 
and the $\kappa$-uniform bounds for $u_i^{(\kappa)}$ 
that $(u_i^{(\kappa)})$ is bounded in $L^{2(d+\beta+1)/d}(0,T;L^{2}(\R^d))$.
Together with estimate \eqref{4.est5}, we conclude that
$(u_i^{(\kappa)})$ is bounded in $L^2(0,T;H^{(\beta+1)/2}(\R^d))$. 

We claim that the embedding $H^{(\beta+1)/2}(\R^d)\cap L^1(\R^d;|x|dx)
\hookrightarrow L^2(\R^d)$ is compact. This claim follows from
\cite[Corollary 7.2]{NPV12}, applied to balls (which are extension domains
due to \cite[Theorem 5.4]{NPV12}), \cite[Lemma 1]{CGZ20}, and a Cantor
diagonal argument. In view of the moment estimate for $u_i^{(\kappa)}$ and
the $L^q(0,T;W^{-1,q}(\R^d))$ bound for $\pa_t u_i^{(\kappa)}$, the Aubin--Lions
lemma yields, up to a subsequence, the convergence
$$
  u_i^{(\kappa)}\to u_i\quad\mbox{strongly in }L^2(0,T;L^2(\R^d))
	\mbox{ as }\kappa\to 0.
$$
All the terms in \eqref{4.ukappa} have been already estimated in Section 
\ref{sec.est} except those depending explicitly on $\kappa$, in particular
$$
  \kappa g_0[u_i^{(\kappa)}](x) = \kappa u_i^{(\kappa)}(x)^2
	- \kappa\frac{e^{-|x|^2}}{\pi^{d/2}}\int_{\R^d}u_i^{(\kappa)}(y)^2 dy.
$$
The strong convergence of $u_i^{(\kappa)}$ in $L^2(\R^d))$ implies that 
$$
  \kappa g_0[u_i^{(\kappa)}] \to 0 \quad\mbox{strongly in }L^1(0,T;L^1(\R^d)).
$$
Therefore, we can perform the limit $\kappa\to 0$ in \eqref{4.ukappa}
to infer that $u$ is a weak solution to \eqref{1.eq}. The entropy inequality
\eqref{4.ei3} for $u^{(\kappa)}$ and Fatou's lemma yield in the 
limit $\eps\to 0$ the entropy inequality \eqref{1.ei} for $u$.
This finishes the proof of Theorem \ref{thm.ex}.

%%%%%%%%%%%%%%%%%%%%%%%%%%%%%%%%%%%%%%%%%%%%%%%%%%%%%%%%%%%%%%%%%%%%%%%%%%%%%%%

\begin{appendix}
\section{Auxiliary results}\label{sec.aux}

We collect some results from fractional calculus used in this paper. 
The following lemma can be found in \cite[Chap.~V, Sect.~1.2]{Ste70}.

\begin{lemma}[Hardy--Littlewood--Sobolev inequality]\label{lem.HLS}
Let $0<s<1$ and $1<p<\infty$. 
Then there exists a constant $C>0$ such that for all $u\in L^p(\R^d)$,
$$
  \|(-\Delta)^{-s}u\|_q\le C\|u\|_p, \quad\mbox{where }
	\frac{1}{p} = \frac{1}{q} + \frac{2s}{d}.
$$
\end{lemma}

The following Stroock--Varopoulos-type inequality is known even for 
general functions; see, e.g., \cite[Lemma 7.2]{STV19}. 

\begin{lemma}[Generalized Stroock--Varopoulos inequality]\label{lem.GSVI}
Let $u\in H^s(\R^d)$ be such that $u\ge 0$ in $\R^d$ and 
$\log(u)(-\Delta)^s u\in L^1(\R^d)$, 
where $0<s<1$. Then
$$
  \int_{\R^d}\log(u)(-\Delta)^s u dx
	\ge 4\int_{\R^d}|(-\Delta)^{s/2}\sqrt{u}|^2 dx.
$$
\end{lemma}

\begin{proof}
A symmetry argument shows that
$$
  \int_{\R^d}\log(u)(-\Delta)^s u dx
	= \frac{c_{d,s}}{2}\int_{\R^d}\int_{\R^d}
	\frac{(u(x)-u(y))(\log u(x)-\log u(y))}{|x-y|^{d+2s}}dxdy.
$$
Elementary computations yield
\begin{align*}
  (u(x)&-u(y))(\log u(x)-\log u(y)) \\
	&= 4\big(\sqrt{u(x)}-\sqrt{u(y)}\big)^2 
	\frac{\sqrt{u(x)}+\sqrt{u(y)}}{2(\sqrt{u(x)}-\sqrt{u(y)})}
	\big(\log \sqrt{u(x)} -\log \sqrt{u(y)}\big).
\end{align*}
We claim that
$$
  \frac{\sqrt{u(x)}+\sqrt{u(y)}}{2(\sqrt{u(x)}-\sqrt{u(y)})}
  \big(\log \sqrt{u(x)} -\log \sqrt{u(y)}\big)\geq 1.
$$
Notice that the above relation holds in the limit $u(x)\to u(y)$ as $x\to y$. 
Therefore, we can assume without loss of generality that $u(x)>u(y)$.
Defining $Z = \log \sqrt{u(x)} -\log \sqrt{u(y)} > 0$, the previous
inequality is equivalent to
$$
  \frac{e^Z + 1}{e^Z - 1}\frac{Z}{2} \geq 1\quad\Longleftrightarrow\quad
  e^Z + 1 \geq \frac{e^Z - 1}{Z/2}.
$$
Multiplying both sides of the inequality by $e^{-Z/2}/2$ yields the 
elementary relation $\sinh(Z/2)\leq(Z/2)\cosh(Z/2)$, which is true.

This gives, using the same symmetry argument as before,
\begin{align*}
  \int_{\R^d}\log u(-\Delta u)^s u dx
	&\ge 2c_{d,s}\int_{\R^d}\int_{\R^d}\frac{(\sqrt{u(x)}-\sqrt{u(y)})^2}{|x-y|^{d+2s}}
	dxdy \\
	&= 4\int_{\R^d}\sqrt{u}(-\Delta)^s\sqrt{u} dx 
	= 4\int_{\R^d}|(-\Delta)^{s/2}\sqrt{u}|^2 dx,
\end{align*}
finishing the proof.
\end{proof}

The following lemma can be proved exactly as in \cite[Lemma 1]{CGZ20}.

\begin{lemma}[Compactness]\label{lem.comp}
Let $d\ge 2$, $1\le p<d$, $m>0$, and $0<r\le p$. Then the space
$$
  \bigg\{v\in W^{1,p}(\R^d):
	\int_{\R^d}(1+|x|^2)^{m/2}|v(x)|^r dx<\infty\bigg\}
$$
is compactly embedded into $L^q(\R^d)$ for any $\max\{1,m\}\le q<dp/(d-p)$.
\end{lemma}

The previous lemma allows us to prove a compactness result in $\R^d$
of Aubin--Lions type.

\begin{lemma}\label{lem.Aubin}
Let $d\ge 1$, $T>0$, $m>0$, $s\ge 0$, and let $(u_\eps)$ be a family of nonnegative
functions satisfying
\begin{align*}
  \|\sqrt{u_\eps}\|_{L^2(0,T;H^1(\R^d))} + \|\pa_t u_\eps\|_{L^1(0,T;H^{-s}(\R^d))}
	&\le C, \\
	\|u_\eps^2(\log u_\eps)_+\|_{L^1(0,T;L^1(\R^d))}
	+ \|(1+|\cdot|^{2})^{m/2} u_\eps\|_{L^\infty(0,T;L^1(\R^d))} &\le C
\end{align*}
for some $C>0$ independent of $\eps>0$. Then, up to a subsequence,
$$
  u_\eps\to u \quad\mbox{strongly in }L^2(0,T;L^{2}(\R^d))\mbox{ as }\eps\to 0.
$$
\end{lemma}

\begin{proof}
The bounds for $(u_\eps)$ imply that $\na u_\eps = 2\sqrt{u_\eps}\na\sqrt{u_\eps}$
is bounded in $L^2(0,T;L^1(\R^d))$ and consequently, $(u_\eps)$ is bounded
in $L^2(0,T;W^{1,1}(\R^d))$. By Lemma \ref{lem.comp},
$V:=\{v\in W^{1,1}(\R^d):\int_{\R^d}(1+|x|^2)^{m/2}v(x)dx<\infty\}$ is
compactly embedded into $L^q(\R^d)$ for any $1\le q<d/(d-1)$. 

If $s\ge d/2$, the embedding $H^s(\R^d)\hookrightarrow L^{q'}(\R^d)$ for
$q'=q/(q-1)>d$ implies that $L^q(\R^d)\hookrightarrow H^{-s}(\R^d)$ is continuous.
Thus, we can apply the standard Aubin--Lions lemma with the spaces $V\hookrightarrow
L^q(\R^d)\hookrightarrow H^{-s}(\R^d)$. If $s<d/2$, it holds that
$H^{d/2}(\R^d)\hookrightarrow H^{s}(\R^d)$ and
$H^{-s}(\R^d)\hookrightarrow H^{-d/2}(\R^d)$ and consequently, $(\pa_t u_\eps)$
is bounded in $L^1(0,T;H^{-d/2}(\R^d))$. In any case, 
the Aubin--Lions lemma can be applied
with $V\hookrightarrow L^q(\R^d)\hookrightarrow H^{-\max\{d/2,s\}}(\R^d)$.
Thus, there exists a subsequence of $(u_\eps)$, which is
not relabeled, such that $u_\eps\to u$ strongly in $L^2(0,T;L^q(\R^d))$
as $\eps\to 0$.

It remains to show that this convergence holds in $L^2(0,T;L^2(\R^d))$. To this
end, we observe that there exists $C>0$ such that
$f(z):=z^2\log(1+z^2)\le C(1+z^2(\log z^2)_+)$ for $z\in\R$, 
and for any $1<q<2$ and $\delta>0$, 
there exists $C(\delta)>0$ such that $z^2\le \delta f(z) + C(\delta)|z|^q$ for $s\in\R$.
Since $f$ is even, increasing on $[0,\infty)$, and convex, 
this gives with $z=(u_\eps-u)/2$ and for any $\delta>0$,
\begin{align*}
  \frac14\|u_\eps-u\|_{L^2(0,T;L^2(\R^d))}^2
	&\le \delta \int_0^T\int_{\R^d}f\bigg(\frac{u_\eps-u}{2}\bigg)dx dt
	+ C(\delta)\|u_\eps-u\|_{L^q(0,T;L^q(\R^d))}^q \\
	&\le \delta \int_0^T\int_{\R^d}f\bigg(\frac{u_\eps+u}{2}\bigg)dx dt
	+ C(\delta)\|u_\eps-u\|_{L^q(0,T;L^q(\R^d))}^q \\
	&\le \frac{\delta}{2}\int_0^T\int_{\R^d}(f(u_\eps)+f(u))dxdt
	+ C(\delta)\|u_\eps-u\|_{L^q(0,T;L^q(\R^d))}^q.
\end{align*}
We apply the limes superior $\eps\to 0$ to both sides and use the strong
convergence of $(u_\eps)$:
$$
  \limsup_{\eps\to 0}\|u_\eps-u\|_{L^2(0,T;L^2(\R^d))}^2 
	\le 2\delta \limsup_{\eps\to 0}\int_0^T\int_{\R^d}
	(f(u_\eps)+f(u))dx dt	\le \delta C.
$$
Since $\delta>0$ is arbitrary, the conclusion follows.
\end{proof}
\end{appendix}

%%%%%%%%%%%%%%%%%%%%%%%%%%%%%%%%%%%%%%%%%%%%%%%%%%%%%%%%%%%%%%%%%%%%%%%%%%%%%%


\begin{thebibliography}{11}
\bibitem{Ama93} H.~Amann. Nonhomogeneous linear and quasilinear elliptic and 
parabolic boundary value problems. In: H.~J.~Schmeisser and H.~Triebel (eds.), 
Funct. Spaces Differ. Oper. Nonlin. Anal., pp. 9--126. Teubner, Wiesbaden, 1993.
%

\bibitem{BGHP85} M.~Bertsch, M.~Gurtin, D.~Hilhorst, and L.~Peletier. On interacting 
populations that disperse to avoid crowding: preservation of segregation. 
{\em J. Math. Biol.} 23 (1985), 1--13.
%
\bibitem{BWZ17} U.~Biccari, M.~Warma, and E.~Zuazua. Local elliptic regularity for the 
Dirichlet fractional Laplacian. {\em Adv. Nonlin. Stud.} 17 (2017), 387--409.

\bibitem{BKM10} P.~Biler, G.~Karch, and R.~Monneau. Nonlinear diffusion of dislocation
density and self-similar solutions. {\em Commun. Math. Phys.} 294 (2010), 145--168.

\bibitem{CGZ20} L.~Caffarelli, M.~Gualdani, and N.~Zamponi. Existence of weak solutions 
to a continuity equation with space time nonlocal Darcy law. 
{\em Commun. Partial Differ. Eqs.} 45 (2020), 1799--1819.
%
\bibitem{CaVa11} L.~Caffarelli and J.~L.~V\'azquez. 
Nonlinear porous medium flow with fractional potential pressure. 
{\em Arch. Ration. Mech. Anal.} 202 (2011), 537--565.
%
\bibitem{CaVa11a} L.~Caffarelli and J.~L.~V\'azquez. 
Asymptotic behaviour of a porous medium equation with fractional diffusion. 
{\em Discrete Contin. Dyn. Syst.} 29 (2011), 1393--1404.
%
\bibitem{CDJ19} L.~Chen, E.~S.~Daus, and A.~J\"ungel. Rigorous mean-field limit and 
cross diffusion. {\em Z. Angew. Math. Phys.} 70 (2019), no.~122, 21 pages.

\bibitem{ChJu21} X.~Chen and A.~J\"ungel. When do cross-diffusion systems have an 
entropy structure? {\em J. Differ. Eqs.} 278 (2021), 60--72.

\bibitem{DPR20} E.~S.~Daus, M.~Ptashnyk, and C.~Raithel. Derivation of a fractional
cross-diffusion system as the limit of a stochastic many-particle system driven by 
L\'evy noise. {\em J. Differ. Eqs.} 309 (2022), 386--426.

\bibitem{DiMo21} H.~Dietert and A.~Moussa. Persisting entropy structure for nonlocal 
cross-diffusion systems. Submitted for publication, 2021.
arXiv:2101.02893.

\bibitem{DiFa13} M.~Di Francesco and S.~Fagioli. Measure solutions for non-local
interaction PDEs with two species. {\em Nonlinearity} 26 (2013), 2777--2808.

\bibitem{NPV12} E.~Di Nezza, G.~Palatucci, and E.~Valdinoci. Hitchhiker's guide to
fractional Sobolev spaces. {\em Bull. Sci. Math.} 136 (2012), 521--573.
%
{
\bibitem{DrJu20} P.-E.~Druet and A.~J\"ungel. Analysis of cross-diffusion systems for 
fluid mixtures driven by a pressure gradient. 
{\em SIAM J. Math. Anal.} 52 (2020), 2179--2197.}
%
\bibitem{Esc06} C.~Escudero. The fractional Keller--Segel model.
{\em Nonlinearity} 19 (2006), 2909--2918.

\bibitem{GaVe19} G.~Galiano and J.~Velasco. Well-posedness of a cross-diffusion 
population model with nonlocal diffusion. 
{\em SIAM J. Math. Anal.} 51 (2019), 2884--2902.

\bibitem{GHLP21} V.~Giunta, T.~Hillen, M.~A.~Lewis, and J.~R.~Potts.
Local and global existence for non-local multi-species
advection-diffusion models. Submitted for publication, 2021. arXiv:2106.06383.

\bibitem{GuZa18} M.~Gualdani and N.~Zamponi. Global existence of weak even solutions 
for an isotropic Landau equation with Coulomb potential. 
{\em SIAM J. Math. Anal.} 50 (2018), 3676--3714. 

\bibitem{IWM21} N.~Iqbal, R.~Wu, and W.~Mohammed. Pattern formation induced by 
fractional cross-diffusion in a 3-species food chain model with harvesting.
{\em Math. Computers Simul.} 188 (2021), 102--119.

\bibitem{Jue16} A.~J\"ungel. {\em Entropy Methods for Diffusive Partial Differential
Equations}. BCAM SpringerBriefs, 2016.

\bibitem{JPZ21} A.~J\"ungel, S.~Portisch, and A.~Zurek. Nonlocal cross-diffusion 
systems for multi-species populations and networks. To appear in {\em Nonlin. Anal.},
2022. arXiv:2104.06292.

\bibitem{JuZu20} A.~J\"ungel and A.~Zurek. A finite-volume scheme for a 
cross-diffusion model arising from interacting many-particle population systems.
In: R.~Kl\"ofkorn, E.~Keilegavlen, F.~Radu, and J.~Fuhrmann (eds.). 
{\em Finite Volumes for Complex Applications IX -- Methods, Theoretical Aspects, 
Examples}. Springer, Cham, 2020, pp.~223-231.

\bibitem{STV19} D.~Stan, F.~del Teso, and J.-L.~V\'azquez. Existence of weak solutions
for a general porous medium equation with nonlocal pressure. 
{\em Arch. Ration. Mech. Anal.} 233 (2019), 451--496.

\bibitem{Ste70} E.~Stein. {\em Singular Integrals and Differentiability Properties 
of Functions}. Princeton University Press, Princeton, 1970.

\end{thebibliography}
\end{document}